\documentclass{amsart}
\usepackage{amssymb, diagrams, enumerate, setspace}
\usepackage{hyperref}
\renewcommand{\AA}{\mathbb{A}}
\newcommand{\BB}{\mathbb{B}}
\newcommand{\CC}{\mathbb{C}}
\newcommand{\DD}{\mathbb{D}}

\newcommand{\GG}{\mathbb{G}}
\newcommand{\II}{\mathbb{I}}
\newcommand{\NN}{\mathbb{N}}
\newcommand{\QQ}{\mathbb{Q}}
\newcommand{\ZZ}{\mathbb{Z}}

\newcommand{\cF}{\mathcal{F}}
\newcommand{\cG}{\mathcal{G}}
\newcommand{\cH}{\mathcal{H}}
\newcommand{\cL}{\mathcal{L}}
\newcommand{\cO}{\mathcal{O}}

\newcommand{\fc}{\mathfrak{c}}
\newcommand{\ff}{\mathfrak{f}}

\newcommand{\fp}{\mathfrak{p}}
\newcommand{\fpb}{{\bar{\fp}}}
\newcommand{\fP}{\mathfrak{P}}

\newcommand{\Qp}{\QQ_p}
\newcommand{\Zp}{\ZZ_p}
\newcommand{\Cp}{\CC_p}
\newcommand{\Qb}{\overline{\QQ}}
\newcommand{\Qpb}{\Qb_p}
\newcommand{\Qpnr}{\QQ_{p}^{\mathrm{nr}}}

\newcommand{\OFinf}{\widehat{\cO}_{F_\infty}}
\newcommand{\Finf}{\widehat{F}_\infty}

\newcommand{\htimes}{\mathbin{\hat\otimes}}
\newcommand{\vp}{\varphi}
\newcommand{\ve}{\varepsilon}
\newcommand{\Dcris}{\mathbb{D}_{\mathrm{cris}}}
\newcommand{\DdR}{\mathbb{D}_{\dR}}
\newcommand{\Brig}{\BB^+_{\rig,\Qp}}

\newcommand{\Bcris}{\BB_{\cris}}

\newcommand{\fq}{\mathfrak{q}}

\DeclareMathOperator{\Gal}{Gal}

\DeclareMathOperator{\Iw}{Iw}

\DeclareMathOperator{\Tr}{Tr}
\DeclareMathOperator{\Tw}{Tw}
\DeclareMathOperator{\pr}{pr}

\DeclareMathOperator{\End}{End}
\DeclareMathOperator{\GL}{GL}
\DeclareMathOperator{\rig}{rig}
\DeclareMathOperator{\cris}{cris}
\DeclareMathOperator{\dR}{dR}
\DeclareMathOperator{\Fil}{Fil}
\DeclareMathOperator{\arith}{arith}

\DeclareMathOperator{\rank}{rank}
\DeclareMathOperator{\Hom}{Hom}

\DeclareMathOperator{\loc}{loc}
\DeclareMathOperator{\Ind}{Ind}
\DeclareMathOperator{\Leop}{Leop}

\newcommand{\onto}{\twoheadrightarrow}
\newcommand{\into}{\hookrightarrow}
\newtheorem{definition}{Definition}[section]
\newtheorem{proposition}[definition]{Proposition}
\newtheorem{theorem}[definition]{Theorem}
\newtheorem{lemma}[definition]{Lemma}
\newtheorem{corollary}[definition]{Corollary}
\newtheorem{conjecture}[definition]{Conjecture}

\newtheorem{assumption}[definition]{Assumption}
\theoremstyle{remark}
\newtheorem{remark}[definition]{Remark}
\newtheorem{example}[definition]{Example}

\begin{document}

\markboth{David Loeffler and Sarah Livia Zerbes}
{Iwasawa theory and p-adic L-functions over $\Zp^2$-extensions}

\title{IWASAWA THEORY AND P-ADIC L-FUNCTIONS OVER $\Zp^2$-EXTENSIONS}

\author{DAVID LOEFFLER}
\thanks{Supported by a Royal Society University Research Fellowship.}
\address{Mathematics Institute\\
Zeeman Building\\
University of Warwick\\
Coventry CV4 7AL, United Kingdom}
\email{d.a.loeffler@warwick.ac.uk}

\author{SARAH LIVIA ZERBES}
\thanks{Supported by EPSRC First Grant EP/J018716/1.}
\address{Mathematics Department\\
  University College London\\
  Gower Street\\
  London WC1E 6BT, United Kingdom}
\email{s.zerbes@ucl.ac.uk}

\begin{abstract}
  We construct a two-variable analogue of Perrin-Riou's $p$-adic regulator map for the Iwasawa cohomology of a crystalline representation of the absolute Galois group of $\Qp$, over a Galois extension whose Galois group is an abelian $p$-adic Lie group of dimension 2.

  We use this regulator map to study $p$-adic representations of global Galois groups over certain abelian extensions of number fields whose localisation at the primes above $p$ is an extension of the above type. In the example of the restriction to an imaginary quadratic field of the representation attached to a modular form, we formulate a conjecture on the existence of a ``zeta element'', whose image under the regulator map is a $p$-adic $L$-function. We show that this conjecture implies the known properties of the 2-variable $p$-adic $L$-functions constructed by Perrin-Riou and Kim.
\end{abstract}



\maketitle

 \section{Introduction}
  \label{sect:intro}

  In the first part of this paper (Sections \ref{sect:yager} and \ref{sect:regulator}), we develop a ``two-variable'' analogue of Perrin-Riou's theory of $p$-adic regulator maps for crystalline representations of $p$-adic Galois groups.

  Let us briefly recall Perrin-Riou's cyclotomic theory as developed in \cite{perrinriou95}. Let $p$ be an odd prime, $F$ a finite unramified extension of $\Qp$, and $V$ a continuous $p$-adic representation of the absolute Galois group $\cG_F$ of $F$, which is crystalline with Hodge--Tate weights $\ge 0$ and with no quotient isomorphic to the trivial representation. Then there is a ``regulator'' or ``big logarithm'' map
  \[ \mathcal{L}^\Gamma_{F, V} : H^1_{\Iw}(F(\mu_{p^\infty}), V) \rTo \cH_{\Qp}(\Gamma) \otimes_{\Qp} \Dcris(V)\]
  which interpolates the values of the Bloch--Kato dual exponential and logarithm maps for the twists $V(j)$, $j \in \ZZ$, over each finite subextension $F(\mu_{p^n})$. Here $\cH_{\Qp}(\Gamma)$ is the algebra of $\Qp$-valued distributions on the group $\Gamma = \Gal(F(\mu_{p^\infty}) / F) \cong \Zp^\times$, and the Iwasawa cohomology $H^1_{\Iw}(F(\mu_{p^\infty}), V)$ is defined as $\Qp \otimes_{\Zp} \varprojlim_n H^1(F(\mu_{p^n}), T)$ where $T$ is any $\cG_F$-stable $\Zp$-lattice in $V$. This map plays a crucial role in cyclotomic Iwasawa theory for $p$-adic representations of the Galois groups of number fields, as a bridge between cohomological objects and $p$-adic $L$-functions.

  It is natural to ask whether or not the construction of the maps $\mathcal{L}^\Gamma_{F, V}$ may be extended to consider twists of $V$ by more general characters of $\cG_{F}$. In this paper, we give a complete answer to this question for characters factoring through an extension $K_\infty / F$ which is abelian over $\Qp$ (thus for all characters if $F = \Qp$). Any such character factors through the Galois group $G$ of an extension of the form $K_\infty = F_\infty(\mu_{p^\infty})$, where $F_\infty$ is an unramified extension of $F$ which is a finite extension of the unique unramified $\Zp$-extension of $F$. Denote by $\Finf$ the $p$-adic completion of $F_\infty$, and $\cH_{\Finf}(G)$ the algebra of $\Finf$-valued distributions on $G$.

  \begin{theorem}
   For any crystalline representation $V$ of $\cG_{F}$ with non-negative Hodge--Tate weights, there exists a regulator map
   \[ \mathcal{L}^{G}_{V} : H^1_{\Iw}(K_\infty, V) \rTo \cH_{\Finf}(G) \otimes_{\Qp} \Dcris(V)\]
   interpolating the maps $\mathcal{L}^\Gamma_{K, V}$ for all unramified extensions $K/F$ contained in $F_\infty$.
  \end{theorem}

  See Theorem \ref{thm:localregulator} for a precise statement of the result. Unlike the cyclotomic case, this result holds whether or not $V$ has trivial quotients.

  In Sections \ref{sect:semilocalregulator} and \ref{sect:defmoduleofLfunctions}, we use the $2$-variable $p$-adic regulator to study global Galois representations. Let $K$ be a finite extension of $\QQ$, $\fp$ a prime of $K$ above $p$ which is unramified, and $K_\infty$ be a $p$-adic Lie extension of $K$ such that for any prime $\fP$ of $K_\infty$ above $\fp$, the local extension $K_{\infty, \fP} / K_{\fp}$ is of the type considered above. Let $G=\Gal(K_\infty\slash K)$. In Section \ref{sect:semilocalregulator}, we extend the regulator map to a map
  \[ \cL^G_{\fp, V} : Z^1_{\Iw, \fp}(K_\infty, V) \rTo \cH_{\Finf}(G) \otimes_{\Qp} \Dcris(K_{\fp}, V) \]
  where $Z^1_{\Iw, \fp}(K_\infty, V)$ is the direct sum of the Iwasawa cohomology groups at each of the primes $\fq \mid \fp$, and $\Dcris(K_{\fp}, V)$ is the Fontaine $\Dcris$ functor for $V$ regarded as a representation of a decomposition group at $\fp$. There is a natural localisation map
  \[ H^1_{\Iw, S}(K_\infty, V) \to  \bigoplus_{\fp \mid p} Z^1_{\Iw, \fp}(K_\infty, V)\]
  where $H^1_{\Iw, S}(K_\infty,V)$ denotes the inverse limit of global cohomology groups unramified outside a fixed set of primes $S$. As in the case of Perrin-Riou's cyclotomic regulator map, our map $\cL^G_V$ allows elements of Iwasawa cohomology (or, more generally, of its exterior powers) to be interpreted as $\Dcris$-valued distributions on $G$ (after extending scalars). Assuming a plausible conjecture analogous to Leopoldt's conjecture, we use the map $\cL^G_V$ to define a certain submodule $\II_{\arith}(V)$ of the distributions on $G$ with values in an exterior power of $\Dcris$. Following Perrin-Riou \cite{perrinriou95}, we call $\II_{\arith}(V)$ the \emph{module of $2$-variable $L$-functions}. We conjecture that there exist special elements of the top exterior power of $H^1_{\Iw,S}(K_\infty, V)$ (``zeta elements'') whose images under the regulator map are $p$-adic $L$-functions, and that these should generate $\II_{\arith}(V)$ as a module over the Iwasawa algebra $\Lambda_{\Qp}(G)$.

  In Section \ref{sect:imquad}, we investigate in detail two instances of this conjecture that occur when the field $K$ is imaginary quadratic. We first show that for the representation $\Zp(1)$, our regulator map coincides with the map constructed in \cite{yager82}. In this paper, Yager shows that his map sends the Euler system of elliptic units to Katz's $p$-adic $L$-function. As the second example, we study the representation attached to a weight 2 cusp form for $\GL_2 / K$: here we predict the existence of multiple distributions, depending on a choice of Frobenius eigenvalue at each prime above $p$ (Conjecture \ref{conj:twovarmf}), and we show that our conjectures imply the known properties of the 2-variable $p$-adic $L$-functions constructed by Perrin-Riou \cite{perrinriou88} (for $f$ ordinary) and by B.D.~Kim \cite{kim-preprint} (for $f$ non-ordinary). However, our conjectures also predict the existence of some new $p$-adic $L$-functions. (The existence of these $p$-adic $L$-functions is verified in a forthcoming paper \cite{loeffler13} of the first author.)

  In our paper \cite{leiloefflerzerbes12} (joint with Antonio Lei), we use the $2$-variable $p$-adic regulator to study the critical slope $p$-adic $L$-functions of an ordinary CM modular form. In this case, there are two candidates for the $p$-adic $L$-function, one arising from Kato's Euler system and a second from $p$-adic modular symbols. The latter has been studied by Bella\"\i che \cite{bellaiche11}, who has proved a formula (Theorem 2 of \emph{op.cit.}) relating it to the Katz $L$-function for the CM field. We use the methods of the present paper to prove a corresponding formula for the $L$-function arising from Kato's construction, implying that the two $p$-adic $L$-functions in fact coincide.


 \section{Setup and notation}


  \subsection{Fields and their extensions}
   \label{sect:fieldsandtheirextensions}

   Let $p$ be an odd prime, and denote by $\mu_{p^\infty}$ the set of $p$-power roots of unity. Let $K$ be a finite extension of either $\QQ$ or $\Qp$. Define the Galois groups $\cG_K=\Gal(\overline{K}\slash K)$ and $H_K=\Gal(\overline{K}\slash K(\mu_{p^\infty}))$. A \emph{$p$-adic Lie extension} of $K$ is a Galois extension $K_\infty/K$ such that $\Gal(K_\infty / K)$ is a compact $p$-adic Lie group of finite dimension.

   We write $\Gamma$ for the Galois group $\Gal(\QQ(\mu_{p^\infty}) / \QQ)\cong \Gal(\Qp(\mu_{p^\infty})\slash \Qp)$, which we identify with $\Zp^\times$ via the cyclotomic character $\chi$. Then $\Gamma \cong \Delta \times \Gamma_1$, where $\Delta$ is cyclic of order $p-1$ and $\Gamma_1 = \Gal(\Qp(\mu_{p^\infty}) / \Qp(\mu_p)) \cong \Zp$, so in particular $\QQ_\infty$ (resp. $\QQ_{p,\infty}$) is a $p$-adic Lie extension of $\QQ$ (resp. $\Qp$) of dimension $1$.


  \subsection{Iwasawa algebras and power series}
   \label{sect:iwasawaalgs}

   Let $G$ be a compact $p$-adic Lie group, and $L$ a complete discretely valued extension of $\Qp$ with ring of integers $\cO_L$. We let $\Lambda_{\cO_L}(G)$ be the Iwasawa algebra $\varprojlim_U \cO_L[G /U]$, where the limit is taken over open subgroups $U \subseteq G$. We shall always equip this with the inverse limit topology (sometimes called the ``weak topology'') for which it is a Noetherian topological $\cO$-algebra (cf.~\cite[Theorem 6.2.8]{emerton04}). If $L / \Qp$ is a finite extension then $\Lambda_{\cO_L}(G)$ is compact (but not otherwise).

   We let $\Lambda_L(G) = L \otimes_{\cO_L} \Lambda_{\cO}(G)$, which is also Noetherian; it is isomorphic to the continuous dual of the space $C(G, L)$ of continuous $L$-valued functions on $G$. (See \cite[Corollary 2.2]{schneiderteitelbaum02} for a proof of the last statement when $L / \Qp$ is a finite extension; this extends immediately to general discretely-valued $L$, since $\Lambda_{L}(G) = L \htimes_{\Qp} \Lambda_{\Qp}(G)$ and similarly for $C(G, L)$.)


   Let $\cH_L(G)$ be the space of $L$-valued locally analytic distributions on $G$ (the continuous dual of the space $C^{\mathrm{la}}(G, L)$ of $L$-valued locally analytic functions on $G$). There is an injective algebra homomorphism $\Lambda_L(G) \into \cH_L(G)$ (see \cite[Proposition 2.2.7]{emerton04}), dual to the inclusion of $C^{\mathrm{la}}(G, L)$ as a dense subspace of $C(G, L)$. We endow $\cH_L(G)$ with its natural topology as an inverse limit of Banach spaces, with respect to which the map $\Lambda_L(G) \into \cH_L(G)$ is continuous.

   We shall mostly be concerned with the case when $G$ is abelian, in which case $G$ has the form $H \times \Zp^d$ for $H$ a finite abelian group. In this case $\Lambda_{\cO_L}(G)$ is isomorphic to the power series ring $\cO_L[H][[X_1, \dots, X_d]]$, where $X_i = \gamma_i - 1$ for generators $\gamma_1, \dots, \gamma_d$ of the $\Zp^d$ factor (see \cite[\S 8.4.1]{nekovar06}). The weak topology on $\Lambda_{\cO_L}(G)$ is the $I$-adic topology, where $I$ is the ideal $(p, X_1, \dots, X_d)$. Meanwhile, $\cH_L(G)$ identifies with the algebra of $L[H]$-valued power series in $X_1, \dots, X_d$ converging on the rigid-analytic unit ball $|X_i| < 1$, with the topology given by uniform convergence on the closed balls $|X_i| \le r$ for all $r < 1$.

   In particular, for the group $\Gamma \cong \Zp^*$ as in Section \ref{sect:fieldsandtheirextensions}, we may identify $\cH_L(\Gamma)$ with the space of formal power series
   \[ \{f \in L[\Delta][[X]]:\text{$f$ converges everywhere on the open unit $p$-adic disc}\},\]
   where $X$ corresponds to $\gamma - 1$ for $\gamma$ a topological generator of $\Gamma_1$; and $\Lambda_L(\Gamma)$ corresponds to the subring of $\cH_L(\Gamma)$ consisting of power series with bounded coefficients. Similarly, we define $\cH_L(\Gamma_1)$ as the subring of $\cH_L(\Gamma)$ defined by power series over $\Qp$, rather than $\Qp[\Delta]$.

   For each $i \in \ZZ$, we define an element $\ell_i \in \cH_{\Qp}(\Gamma_1)$ by
   \[ \ell_i = \frac{\log \gamma}{\log \chi(\gamma)} - i\]
   for any non-identity element $\gamma \in \Gamma_1$ (cf.~\cite[\S II.1]{berger03}); note that this differs by a sign from the element denoted by the same symbol in \cite{perrinriou94}.


  \subsection{Fontaine rings}
   \label{sect:Fontainerings}

   We review the definitions of some of Fontaine's rings that we use in this paper. Details can be found in \cite{berger04} or \cite{leiloefflerzerbes11}. Let $K$ be a finite extension of $\QQ_p$; the rings we shall require are those denoted by $\AA_{K}$, $\AA^+_K$, $\BB_K$, $\BB^+_K$, and $\BB^+_{\rig, K}$.

   These rings have intrinsic definitions independent of any choices and valid for any $K$; but we shall be interested in the case when $K$ is unramified over $\QQ_p$. In this case, they have concrete (but slightly noncanonical) descriptions as follows. A choice of compatible system $(\zeta_n)_{n \ge 0}$ of $p$-power roots of unity defines an element $\pi \in \AA_K^+$, and allows us to identify $\AA_K^+$ with the formal power series ring $\cO_K[[\pi]]$. The ring $\AA_K$ is simply $\widehat{\AA^+_K[1/\pi]}$. The ring $\BB^+_K$ is defined as $\AA^+_K[1/p]$, and similarly $\BB_K = \AA_K[1/p]$. Finally, we let $\BB^+_{\rig, K}$ be the ring of power series $f \in K[[\pi]]$ which converge on the open unit disc $|\pi| < 1$.

   All these rings are endowed with an $\cO_K$-linear action of $\Gamma$ by $\gamma(\pi)=(\pi+1)^{\chi(\gamma)}-1$, and with a Frobenius $\vp$ which acts as the usual arithmetic Frobenius on $\cO_K$ and on $\pi$ by $\vp(\pi)=(\pi+1)^p-1$. There is also a left inverse $\psi$ of $\vp$ on all of the above rings, satisfying
   \[
    \vp\circ\psi(f(\pi))=\frac{1}{p}\sum_{\zeta^p=1}f(\zeta(1+\pi)-1).
   \]

   Write $t=\log(1+\pi) \in \Brig$, and $q=\vp(\pi)/\pi\in\AA_{\Qp}^+$. A formal power series calculation shows that $g(t) = \chi(g) t$ for $g \in \Gamma$, and $\vp(t) = pt$.

   The action of $\Gamma$ on $\AA^+_{K}$ gives an isomorphism of $\Lambda_{\cO_K}(\Gamma)$ with the submodule $(\AA^+_{K})^{\psi=0}$, the so-called ``Mellin transform''
   \begin{align*}
    \mathfrak{M}: \Lambda_{\cO_K}(\Gamma) & \rightarrow (\AA^+_{K})^{\psi=0} \\
    f(\gamma-1) & \mapsto f(\gamma-1) \cdot (\pi+1).
   \end{align*}
   This extends to bijections $\Lambda_K(\Gamma) \cong (\BB^+_{K})^{\psi=0}$ and $\cH_K(\Gamma) \cong (\BB^+_{\rig, K})^{\psi=0}$. (See \cite[\S 1.3]{perrinriou90}, \cite[Proposition 1.2.7]{perrinriou94}, or \cite[\S 1.C.2]{leiloefflerzerbes11} for more details.)


  \subsection{Crystalline and de Rham representations}
   \label{sect:crystallinereps}

   Let $K$ be a finite extension of $\Qp$, and $V$ a continuous representation of $\cG_K$ on a $\Qp$-vector space of dimension $d$. Recall that $\DdR(V)$ denotes the space $(V \otimes_{\Qp} \BB_{\dR})^{\cG_K}$, where $\BB_{\dR}$ is Fontaine's ring of periods. This space $\DdR(V)$ is a filtered $K$-vector space of dimension $\le d$, and we say $V$ is \emph{de Rham} if equality holds. If $j \in \ZZ$, $\Fil^j \DdR(V)$ denotes the $j$-th step in the Hodge filtration of $\DdR(V)$.

   If $L$ is a finite extension of $K$, we shall sometimes write $\DdR(L, V)$ for $\DdR(V|_{\cG_L})$, which can be canonically identified with $L \otimes_K \DdR(V)$.

   We also consider the crystalline period ring $\Bcris \subset \BB_{\dR}$, and define similarly $\Dcris(V) = (V \otimes_{\Qp} \Bcris)^{\cG_K}$. This is a $K_0$-vector space of dimension $\le d$, where $K_0$ is the maximal unramified subspace of $K$, endowed with a semilinear Frobenius (acting as the usual arithmetic Frobenius on $K_0$). We say $V$ is \emph{crystalline} if $\dim_{K_0} \Dcris(V) = d$, in which case $V$ is automatically de Rham, and there is a canonical isomorphism of $K$-vector spaces $\DdR(V) \cong K \otimes_{K_0} \Dcris(V)$. As above, we will write $\Dcris(L, V)$ for $\Dcris(V|_{\cG_L})$, where $L$ is a finite extension of $K$; if $V$ is crystalline over $K$ this is isomorphic to $L_0 \otimes_{K_0} \Dcris(V)$.

   For an integer $j$, $V(j)$ denotes the $j$-th Tate twist of $V$, i.e. $V(j)=V\otimes_{\Zp} (\varprojlim_n \mu_{p^n})^{\otimes j}$. If $\zeta = (\zeta_n)_{n \ge 0}$ is a choice of a compatible system of $p$-power roots of unity, this defines a basis vector $e_j$ of $\Qp(j)$ and an element $t^{-j} \in \BB_{\dR}$; these each depend on $\zeta$, but the element $t^{-j} e_j \in \DdR(\Qp(j))$ does not, and tensoring with $t^{-j}e_j$ thus gives a canonical isomorphism $\DdR(\Qp(j)) \cong \Qp$ for each $j$.

   We write
   \[ \exp_{K,V} : \frac{\DdR(V)}{\Fil^0 \DdR(V) + \Dcris(V)^{\vp=1}} \rInto H^1(K,V)\]
   for the \emph{Bloch-Kato exponential} of $V$ over $K$ (c.f. \cite{blochkato90}), which is the boundary map in the cohomology of the ``fundamental exact sequence''
   \[ 0 \rTo V \rTo V \otimes_{\Qp} \Bcris^{\vp = 1} \rTo V \otimes_{\Qp} \left( \frac{\BB_{\dR}}{\BB_{\dR}^+}\right) \rTo 0. \]
   The image of this map is denoted $H^1_e(K,V)$, and we denote its inverse by
   \[ \log_{K, V} : H^1_e(K, V) \rTo^\cong \frac{\DdR(V)}{\Fil^0 \DdR(V) + \Dcris(V)^{\vp=1}}.\]
   We also denote by
   \[ \exp_{K,V}^*: H^1(K,V^*(1)) \rightarrow \Fil^0\DdR(V^*(1))\]
  the \emph{dual exponential} map, which is the dual of $\exp_{K, V}$ with respect to the Tate duality pairing (c.f.~\cite[\S II.1.4]{kato93}); it satisfies the identity
  \[ \langle \exp_{K, V}(a), b \rangle_{K} = \langle a, \exp^*_{K, V}(b)\rangle_{\dR}\]
  for all $a \in \DdR(V)$ and $b \in H^1(K, V)$, where $\langle -, - \rangle_{\mathrm{Tate}}$ is the Tate pairing and $\langle-, -\rangle_{\dR, K}$ is the pairing
  \[ \DdR(V) \otimes \DdR(V^*(1)) \rTo \DdR(\Qp(1)) \cong K \rTo^{\operatorname{trace}} \Qp.\]

  Finally, if $L$ is a number field, $V$ is a $p$-adic representation of $\cG_L$ and $\fp$ is a prime of $L$ above $p$, we write $\DdR(L_\fp, V)$ and $\Dcris(L_\fp, V)$ for the Fontaine spaces attached to $V$ regarded as a representation of $\Gal(\overline{L}_{\fP} / L_\fp)$ for any choice of prime $\fP \mid \fp$ of $\overline{L}$; up to a canonical isomorphism these spaces are independent of the choice of $\fP$.


  \subsection{$(\vp,\Gamma)$-modules and Wach modules}
   \label{sect:phigammawach}

   Let $K$ be a finite extension of $\Qp$, and let $T$ be a $\Zp$-representation of $\cG_K$ (that is, a finite-rank free module over $\Zp$ with a continuous action of $\cG_K$). Denote the $(\vp,\Gamma)$-module of $T$ by $\DD_K(T)$. This is a module over Fontaine's ring $\AA_K$.

   If $K$ is unramified over $\Qp$ and $T$ is a $\Zp$-representation of $\cG_K$ which is crystalline (i.e.~such that $V = T[1/p]$ is crystalline), Wach and Berger have shown that there exists a canonical $\AA^+_K$-submodule $\NN_K(T) \subset \DD_K(T)$, the \emph{Wach module} (see \cite{wach96}, \cite{berger04}); this is the unique submodule such that
   \begin{itemize}
    \item $\NN_K(T)$ is free of rank $d$ over $\AA^+_K$,
    \item the action of $\Gamma$ preserves $\NN_K(T)$ and is trivial on $\NN_K(T) / \pi \NN_K(T)$,
    \item there exists $b \in \ZZ$ such that $\vp(\pi^b \NN_K(T)) \subseteq \pi^b \NN_K(T)$ and the quotient $\pi^b \NN_K(T) / \vp^*(\pi^b \NN_K(T))$ is killed by a power of $q = \vp(\pi)/\pi$.
   \end{itemize}
   Here $\vp^*(\pi^b \NN_K(T))$ denotes the $\AA^+_K$-submodule of $\DD_K(T)$ generated by $\vp(\pi^b \NN_K(T))$.

   The following lemma is immediate from the definition of the functors $\DD_K(-)$ and $\NN_K(-)$:

   \begin{lemma}
    \label{lemma:unramifiedphiGammabasechange}
    Assume that $T$ is a $\Zp$-representation of $\cG_{K}$, and $L$ a finite extension of $K$, with $L$ and $K$ both unramified over $\Qp$. There is a canonical isomorphism of $(\varphi, \Gamma)$-modules
    \[ \DD_L(T) \cong \DD_K(T)\otimes_{\cO_K}\cO_L,\]
    where $\vp$ acts on $\cO_L$ via the arithmetic Frobenius $\sigma_p \in \Gal(L/\Qp)$. If $V = T[1/p]$ is crystalline, then this isomorphism restricts to an isomorphism
    \[ \NN_L(T)\cong \NN_K(T)\otimes_{\cO_K}\cO_L.\]
   \end{lemma}


  \subsection{Iwasawa cohomology and the Perrin-Riou pairing}
   \label{sect:iwasawacoho}

   Let $K$ be a finite extension of $\QQ_\ell$ for some prime $\ell$ (which may or may not equal $p$) and let $T$ be a $\Zp$-representation of $\cG_K$. Let $K_\infty$ be a $p$-adic Lie extension of $K$.

   \begin{definition}
    \label{def:Iwasawacohomology}
    We define
    \[ H^i_{\Iw}(K_\infty, T) := \varprojlim H^i(L, T),\]
    where $L$ varies over the finite extensions of $K$ contained in $K_\infty$, and the inverse limit is taken with respect to the corestriction maps.

    If $V = \Qp \otimes_{\Zp} T$, we write
    \[ H^i_{\Iw}(K_\infty, V) := \Qp \otimes_{\Zp} H^i_{\Iw}(K_\infty,T)\]
    (which is independent of the choice of $\Zp$-lattice $T \subset V$).
   \end{definition}

   It is clear that the groups $H^i_{\Iw}(K_\infty, T)$ are $\Lambda_{\Zp}(G)$-modules; we show in \S \ref{appendix:iwacoho} below that they are finitely generated.

   There is a natural extension of the Tate pairing to this setting. We may clearly choose an increasing sequence $\{K_n\}$ of finite extensions of $K$ with $\bigcup_n K_n = K_\infty$ and each $K_n$ Galois over $K$. If $\langle -, - \rangle_{K_n}$ denotes the Tate pairing $H^1(K_n, T) \times H^1(K_n, T^*(1)) \to \Zp$, and $x = (x_n)$ and $y = (y_n)$ are sequences in $H^1_{\Iw}(K_\infty, T)$ and $H^1_{\Iw}(K_\infty, T^*(1))$, then the sequence whose $n$-th term is
   \begin{equation}
    \label{eq:PRpairing}
    \sum_{\sigma \in \Gal(K_n / K)} \langle x_n, \sigma(y_n) \rangle_{K_n} [\sigma] \in \Zp[\Gal(K_n / K)]
   \end{equation}
   is compatible under the natural projection maps, and hence defines an element of $\Lambda_{\Zp}(G)$.

   \begin{definition}
    We define the \emph{Perrin-Riou pairing} to be the pairing
    \[ \langle - , - \rangle_{K_\infty, T} : H^1_{\Iw}(K_\infty, T) \times H^1_{\Iw}(K_\infty, T^*(1)) \to \Lambda_{\Zp}(G)\]
    defined by the inverse limit of the pairings \eqref{eq:PRpairing}.
   \end{definition}

   It is easy to see that for $\alpha, \beta \in G$ we have
   \[ \langle \alpha x , \beta y \rangle_{K_\infty, T} = \alpha \cdot  \langle x , y \rangle_{K_\infty, T} \cdot \beta^{-1}.\]
   (The above construction is valid for any $p$-adic Lie extension $K_\infty / K$, but in this paper we shall only use the above construction when $G$ is abelian, in which case the distinction between left and right multiplication is not significant.)

   \begin{lemma}
    \label{lemma:PRpairingtwist}
    If $\eta$ is any continuous $\Zp$-valued character of $G$, and we identify $H^1_{\Iw}(K_\infty, T(\eta))$ with $H^1_{\Iw}(K_\infty, T)(\eta)$, then we have
    \[ \langle x, y \rangle_{K_\infty, T(\eta)} = \Tw_{\eta^{-1}} \langle x, y \rangle_{K_\infty, T},\]
    where $\Tw_{\eta}$ is the map $\Lambda_{\Zp}(G) \to \Lambda_{\Zp}(G)$ mapping $g \in G$ to $\eta(g) g$.
   \end{lemma}

   \begin{proof}
    This is immediate if $\eta$ has finite order, and follows for all $\eta$ by reduction modulo powers of $p$; cf.~\cite[\S 3.6.1]{perrinriou94}.
   \end{proof}

   If $V = T[1/p]$, we obtain by extending scalars a pairing
   \[ H^1_{\Iw}(K_\infty, V) \times H^1_{\Iw}(K_\infty, V^*(1)) \to \Lambda_{\Qp}(G)\]
   which we denote by $\langle -, - \rangle_{K_\infty, V}$. This pairing is independent of the choice of lattice $T \subseteq V$.

   It is clear that if $T$ is an $\cO_E$-module for some finite extension $E / \Qp$, then we may similarly define an $\cO_E$-linear analogue of the Perrin-Riou pairing, and in this case Lemma \ref{lemma:PRpairingtwist} applies to any $\cO_E$-valued character $\eta$.

  \subsection{The Fontaine isomorphism}

   In the case when $K_\infty=K(\mu_{p^\infty})$, we can describe $H^1_{\Iw}(K_\infty,T)$ in terms of the $(\vp, \Gamma)$-module $\DD_K(T)$. Let $\Gamma_K = \Gal(K(\mu_{p^\infty}) / K)$, which we identify with a subgroup of $\Gamma$. The following result is originally due to Fontaine (unpublished); for a reference see \cite[Section II]{cherbonniercolmez99}.

   \begin{theorem}
    We have a canonical isomorphism of $\Lambda_{\Zp}(\Gamma_K)$-modules
    \begin{equation}
     \label{eq:fontaineisom1}
     h^1_{\Iw,T}: \DD_K(T)^{\psi=1} \rTo^\cong H^1_{\Iw}(K(\mu_{p^\infty}),T).
    \end{equation}
   \end{theorem}

   If $T$ is a representation of $\cG_{\Qp}$, then the action of $\Gamma$ extends to an action of $\Gal(K(\mu_{p^\infty}) / \Qp)$ on both sides of equation \eqref{eq:fontaineisom1}, and the map $h^1_{\Iw, T}$ commutes with the action of this larger group. We shall apply this below in the case when $K$ is an unramified extension of $\Qp$, so $\Gamma_K = \Gamma$ and $\Gal(K(\mu_{p^\infty}) / \Qp) = \Gamma \times U_F$, where $U_F = \Gal(F / \Qp)$.

   Now let $K$ be a finite unramified extension of $\Qp$, and assume that $V$ is a crystalline representation of $\cG_K$ whose Hodge-Tate weights\footnote{In this paper we adopt the convention that the Hodge--Tate weight of the cyclotomic character is $+1$.} lie in the interval $[a,b]$. The following result is due to Berger \cite[Theorem A.2]{berger03}.

   \begin{theorem}
    \label{thm:crystallinepsiinvariants}
    We have $\DD_K(T)^{\psi=1}\subset \pi^{a-1}\NN_K(T)$. Moreover, if $V$ has no quotient isomorphic to $\Qp(a)$, then $\DD_K(T)^{\psi=1}\subset \pi^a\NN_K(T)$.
   \end{theorem}

   In particular, if $V$ has non-negative Hodge--Tate weights and no quotient isomorphic to $\Qp$, we have $\NN_K(T)^{\psi=1} = \DD_K(T)^{\psi=1}$. Then \eqref{eq:fontaineisom1} becomes
    \begin{equation}
     \label{eq:fontaineisom2}
     h^1_{\Iw,T}: \NN_K(T)^{\psi=1} \rTo^\cong H^1_{\Iw}(K(\mu_{p^\infty}),T).
    \end{equation}


  \subsection{Gauss sums, L- and epsilon-factors}
   \label{sect:epsfactors}

   In many of our formulae, epsilon-factors attached to characters of the Galois group (or rather the Weil group) of $\Qp$ will make an appearance, so we shall fix normalizations for these. We follow the conventions of \cite{deligne73}.

   Let $E$ be an algebraically closed field of characteristic 0, and let $\zeta = (\zeta_n)_{n \ge 0}$ be a choice of a compatible system of $p$-power roots of unity in $E$. The data of such a choice is equivalent to the data of an additive character $\lambda: \Qp \to E^\times$ with kernel $\Zp$, defined by $\lambda(1/p^n) = \zeta_n$.

   We first define the Gauss sum of a finitely ramified character $\omega$ of the Weil group $W_{\mathbb{Q}_p}$, which will in fact depend only on the restriction of $\omega$ to the inertia subgroup $\Gal(\Qpb / \Qpnr)$. If $\omega$ has conductor $n$, then we define
   \[ \tau(\omega, \zeta) = \sum_{\sigma \in \Gal(\Qpnr(\mu_{p^n}) / \Qpnr)} \omega(\sigma)^{-1} \zeta_{p^n}^\sigma.\]

   Now let us recall the definition of epsilon-factors given in \cite{deligne73} for locally constant characters of $\Qp^\times$. These depend on the character $\omega$, the auxilliary additive character $\lambda$, and a choice of Haar measure $\mathrm{d}x$; we choose $\mathrm{d}x$ so that $\Zp$ has volume 1. The definition is given as
   \[ \ve(\omega, \lambda, \mathrm{d}x) =
    \begin{cases}
     1 & \text{if $\omega$ is unramified,} \\
     \int_{\Qp^\times} \omega(x^{-1}) \lambda(x)\, \mathrm{d}x & \text{if $\omega$ is ramified.}
    \end{cases}
   \]
   As shown in \textit{op.cit.}, if the conductor of $\omega$ is $n$, then it suffices to take the integral over $p^{-n} \Zp^\times$. For consistency with \cite{cfksv}, we will rather work with the additive character $\lambda(-x)$ rather than $\lambda(x)$; then we find that
   \[ \ve(\omega, \lambda(-x), \mathrm{d}x) = \omega(p)^n \sum_{x \in (\ZZ / p^n \ZZ)^\times} \omega(x)^{-1} \zeta_n^{-x}.\]

   We now recall that local reciprocity map $\operatorname{rec}_{\Qp}$ of class field theory identifies $W_{\Qp}^{\mathrm{ab}}$ with $\QQ_p^\times$. Following \cite{deligne73}, we normalize $\operatorname{rec}_{\Qp}$ such that \emph{geometric} Frobenius elements of $W_{\Qp}^\mathrm{ab}$ are sent to uniformizers. Then the restriction of $\operatorname{rec}_{\Qp}$ to $\Gal(\QQ_p^\mathrm{ab} / \Qpnr)$ gives an isomorphism
   \[ \Gal(\QQ_p^\mathrm{ab} / \Qpnr) \rTo \Zp^\times.\]
   Our choice of normalization for the local reciprocity map implies that this coincides with the cyclotomic character. On the other hand, $p \in \QQ_p^\times$ corresponds to $\tilde\sigma_p^{-1}$, where $\tilde\sigma_p$ is the unique element of $\Gal(\Qp^\mathrm{ab} / \Qp)$ which acts as the arithmetic Frobenius $\sigma_p$ on $\Qpnr$ and acts trivially on all $p$-power roots of unity. Hence
   \[ \ve(\omega^{-1}, \lambda(-x), \mathrm{d}x) = \omega(\tilde\sigma_p)^{n} \sum_{\sigma \in \Gamma / \Gamma_n} \omega(\sigma) \zeta_n^{-\sigma} = \frac{p^n \omega(\tilde\sigma_p)^n}{\tau(\omega, \zeta)}.\]
   This quantity $\ve(\omega^{-1}, \lambda(-x), \mathrm{d}x)$, which we shall abbreviate to $\ve(\omega^{-1})$, will appear in our formulae for the two-variable regulator.

   We shall also need to consider the case when $E$ is a $p$-adic field and $\omega$ is a continuous character of $\cG_{\Qp}^{\mathrm{ab}}$ which is Hodge--Tate, but not necessarily finitely ramified. Any such character is potentially crystalline, and a well-known construction of Fontaine \cite{fontaine94c} allows us to regard $\DD_{\mathrm{pst}}(\omega)$ as a one-dimensional representation of the Weil group; concretely, if $\omega = \chi^j \omega'$ where $\omega'$ is finitely ramified, then $\sigma \in W_{\Qp}$ acts on $\DD_{\mathrm{pst}}(\omega)$ as $p^{j n(\sigma)} \omega'(\sigma)$, where $n(\sigma)$ is the power of the arithmetic Frobenius by which $\sigma$ acts on $\Qpnr$. We define $\ve(\omega) = \ve(\omega, \lambda(-x), \mathrm{d}x)$ to be the epsilon-factor attached to $\DD_{\mathrm{pst}}(\omega)$, so
   \[ \ve(\omega^{-1}) = \frac{p^{n(1+j)} \omega(\tilde\sigma_p)^n}{\tau(\omega, \zeta)}.\]

   We write $P(\omega, X)$ for the $L$-factor of the Weil--Deligne representation $\DD_{\mathrm{pst}}(\omega)$. This is a polynomial $P(\omega, X)$ in $X$, which is identically 1 if $\omega$ is not crystalline; otherwise, it is given by $P(\omega, X) = 1 - u X$, where $u$ is the scalar by which crystalline Frobenius acts on $\Dcris(\omega)$, so $u = p^{-j} \omega'(\sigma_p)^{-1}$ if $\omega = \chi^j \omega'$ with $\omega'$ unramified.


 \section{Local theory: Yager modules and Wach modules}
  \label{sect:yager}

  \subsection{Some cohomological preliminaries}

   Let $F$ be a finite unramified extension of $\Qp$, and let $F_\infty / F$ be an unramified $p$-adic Lie extension with Galois group $U$. (Thus $U$ is either a finite cyclic group, or the product of such a group with $\Zp$.) Let $\OFinf$ be the completion of the ring of integers of $F_\infty$.

   \begin{lemma}
    \label{lemma:ranktwistedmodule}
    Let $M$ be a free $\Zp$-module of rank $d < \infty$, with a continuous action of $U$. Then the module
    \[ H^0(U, \OFinf \otimes_{\Zp} M)\]
    is free of rank $d$ over $\cO_F$, and
    \[ H^1(U, \OFinf \otimes M) = 0.\]
   \end{lemma}

   \begin{proof}
    This is a form of Hilbert's Theorem 90; for the form of the statement given here see e.g.~\cite[Proposition 1.2.4]{fontaine90}.
   \end{proof}

   We will need the following result on trace maps for unramified extensions.

   \begin{proposition}
    \label{prop:iwasawastructure}
    The module
    \[ \varprojlim_{K} \cO_{K}, \]
    where $K$ varies over finite extensions of $F$ contained in $F_\infty$ and the inverse limit is with respect to the trace maps, is free of rank 1 over $\Lambda_{\cO_F}(U)$.
   \end{proposition}

   \begin{proof}
    We first note that if $L / K$ is any finite unramified extension of local fields, then the trace map $\cO_{L} \to \cO_K$ is surjective, since the residue extension $k_{L} / k_K$ is separable and hence its trace map is surjective. Moreover, $\cO_L$ is free of rank 1 over $\cO_K[\Gal(L/K)]$; elements of $\cO_{L}$ that generate it as a $\cO_F[\Gal(L/K)]$-module are called \emph{integral normal basis generators} of $L/K$. We must show that there exists a trace-compatible sequence $x = (x_K) \in \varprojlim_{K} \cO_{K}$ such that $x_K$ is an integral normal basis generator of $K/F$ for all $K$.

    Let $F_0$ be the largest subfield of $F_\infty$ such that $[F_0 : F]$ is prime to $p$; this is a finite extension of $F$, by our hypotheses on $F_\infty$. Choose a normal basis generator $x_0$ of $F_0 / F$.

    We claim that if $K$ is any finite extension of $F_0$ contained in $F_\infty$, and $x$ is any element of $\cO_K$ with $\operatorname{Tr}_{K/F_0}(x) = x_0$, then $x$ is an integral normal basis generator of $K/F$.

    To prove this, consider the group ring $R = \cO_F[\Gal(K/F)]$. As noted above, $\cO_K$ is a free $R$-module of rank 1. Let $I$ be the ideal of $R$ given by the kernel of the natural map $\cO_F[\Gal(K/F)] \to \cO_F[\Gal(F_0/F)]$. Then $I$ is contained in the Jacobson radical $J$ of $R$ (indeed $J$ is generated by $I$ and $p$). So, by Nakayama's lemma, an element $x \in \cO_K$ generates $\cO_K$ as an $R$-module if and only if its image in $\cO_K / I \cO_K$ generates this quotient; but the trace map $\Tr_{K/F_0}: \cO_K \to \cO_{F_0}$ is surjective and factors through $\cO_K / I \cO_K$, and $\cO_{F_0}$ and $\cO_K / I \cO_K$ are free $\Zp$-modules of the same rank, so $\Tr_{K/F_0}$ must give an isomorphism $\cO_K / I \cO_K \to \cO_{F_0}$. This proves the claim.

    So it suffices to take any element of $\varprojlim_{K} \cO_{K}$ lifting $x_0$.
   \end{proof}

   \begin{remark}
    As noted in \cite{pickett10}, one can also deduce the above claim from the work of Semaev \cite[Lemma 4.1]{semaev88} on normal bases of extensions of \emph{finite} fields, which does not explicitly use Nakayama's lemma.
   \end{remark}


  \subsection{The Yager module}

   In this section we develop a variant of the construction in~\cite[\S 2]{yager82} in order to construct a certain module which, in a sense we shall make precise below, encodes the periods for the unramified characters of $\cG_{F}$.

   \begin{definition}
    Let $K / F$ be a finite unramified extension. For $x \in \cO_K$, we define
    \[ y_{K/F}(x) = \sum_{\sigma \in \Gal(K / F)} x^\sigma\, [\sigma^{-1}] \in \cO_{K}[\Gal(K / F)].\]
   \end{definition}

   It is clear that $y_{K/F}$ is $\cO_F$-linear and injective, and we have $y_{K/F}(x^g) = [g] y_{K/F}(x)$ for all $g \in \Gal(K / F)$, where $[u]$ is the image of $u$ in the group ring. Moreover, the image of $y_{K/F}$ is precisely the submodule $S_{K/F}$ of $\cO_{K}[\Gal(K / F)]$ consisting of elements satisfying $y^g = [g] y$ for all $g \in \Gal(K / F)$, where $y^g$ denotes the action of $\Gal(K / F)$ on the coefficients $\cO_K$.

   \begin{proposition}
    \label{prop:reduction}
    If $L \supset K \supset F$ are finite unramified extensions and $x \in \cO_{L}$, the image of $y_{L/F}(x)$ under the reduction map
    \[ \cO_{L}[\Gal(L / F)] \to \cO_{L}[\Gal(K / F)] \]
    induced by the surjection $\Gal(L / F) \to \Gal(K / F)$ is equal to $y_{K/F}(\operatorname{Tr}_{L/K} x)$. In particular, the reduction has coefficients in $\cO_{K}$.
   \end{proposition}

   \begin{proof}
    Clear from the formula defining the maps $y_{K/F}$ and $y_{L/F}$.
   \end{proof}

   Now let $F_\infty / F$ be any unramified $p$-adic Lie extension with Galois group $U$, as in the previous section. Passing to inverse limits with respect to the trace maps, we deduce that there is an isomorphism of $\Lambda_{\cO_F}(U)$-modules
   \begin{equation}
    \label{def:y}
    y_{F_\infty / F} : \varprojlim_{F \subseteq K \subseteq F_\infty} \cO_{K} \rTo^\cong S_{F_\infty / F} := \varprojlim_{F \subseteq K \subseteq F_\infty} S_{K/F}.
   \end{equation}

   \begin{proposition}
    We have
    \[ S_{F_\infty / F} = \{ f \in \Lambda_{\OFinf}(U) : f^u = [u] f\}\]
    for any topological generator $u$ of $U$.
   \end{proposition}

   \begin{proof}
    Let us set $X = \{ f \in \Lambda_{\OFinf}(U) : f^u = [u] f\}$. Let $F_n$ be a family of finite extensions of $F$ whose union is $F_\infty$, and let $U_n = \Gal(F_\infty / F_n)$.

    Firstly, since $S_{F_n / F} \subseteq \cO_{F_n}[U / U_n] \subseteq \OFinf[U / U_n]$, we clearly have an embedding $S_{F_\infty / F} \into \Lambda_{\OFinf}(U)$, which must land in $X$, because of the Galois-equivariance property of the elements of $S_{F_n / F}$. However, it is clear that for any $x \in X$, the image $x_n$ of $x$ in $\OFinf[U / U_n]$ has coefficients in $\cO_{F_n}$ (since $(\OFinf)^{U_n} = \cO_n$ by Lemma \ref{lemma:ranktwistedmodule}) and satisfies $(x_n)^u = [u] x_n$, thus lies in $S_{F_n}$. So the map $S_{F_\infty / F} \into X$ is a bijection.
   \end{proof}

   We shall always equip $S_{F_\infty / F}$ with the inverse limit topology (arising from the $p$-adic topology of the finitely generated $\Zp$-modules $S_{F_n / F}$). This topology is compact and Hausdorff, and coincides with the subspace topology from $\Lambda_{\OFinf}(U)$.

   \begin{definition}
    We refer to $S_{F_\infty / F}$ as the \emph{Yager module}, since it is closely related to the objects appearing in \cite[\S 2]{yager82}.
   \end{definition}

   We now explain the relation between $S_{F_\infty / F}$ and the periods for characters of $U$. Let $M$ be a finite-rank free $\Zp$-module with an action of $U$, given by a continuous map $\rho: U \to \operatorname{Aut}_{\Zp}(M)$. Then $\rho$ induces a ring homomorphism $\Lambda_{\OFinf}(U) \to \OFinf \otimes_{\Zp} \End_{\Zp} M$, which we also denote by $\rho$.

   \begin{proposition}
    \label{prop:twist}
    Let $\omega \in S_{F_\infty / F}$. Then $\rho(\omega) \in \OFinf \otimes_{\Zp} \End_{\Zp}(M)$ is a period for $\rho$, in the sense that
    \[
      \rho(\omega)^u = \rho(u) \cdot \rho(\omega) .
    \]
    for all $u \in U$.
   \end{proposition}

   \begin{proof}
    Since $\omega \in S_{F_\infty / F}$, we have $\omega^u = [u] \omega$ for any $u \in U$. However, the map $\Lambda_{\OFinf}(U) \to \OFinf \otimes_{\Zp} \End_{\Zp} M$ commutes with the action of $U$ on the coefficient ring $\OFinf$; so we have
    \[ \rho(\omega)^u = \rho(\omega^u) = \rho([u] \cdot \omega) = \rho(u) \rho(\omega).\]
   \end{proof}

   \begin{remark}
    After the results in this section had been proven, we discovered that similar results had been obtained by Pasol in his unpublished PhD thesis \cite[\S 2.5]{pasol05}. Our module $S_{F_\infty / F}$ is the same as his module $\DD_0$. He uses the module $\DD_0$ to relate Katz's $2$-variable $p$-adic $L$-functions attached to a CM elliptic curve to the modular symbols construction by Greenberg and Stevens  \cite{greenbergstevens93}.
   \end{remark}


  \subsection{P-adic representations}

   Let $T$ be a crystalline $\Zp$-representation of $\cG_{F}$. If $K / F$ is any unramified extension, we have isomorphisms $\NN_K(T) \cong \NN_F(T) \otimes_{\cO_F} \cO_K$, so we have trace maps $\NN_L(T) \to \NN_K(T)$ for $L / K$ any two finite unramified extensions of $F$.

   \begin{definition}
    \label{def:Dinfty}
    Let $\NN_{F_\infty}(T) = \varprojlim_{F \subseteq K \subset F_\infty} \NN_K(T)$, where the inverse limit is taken with respect to the trace maps.
   \end{definition}

   By construction, $\NN_{F_\infty}(T)$ has actions of $\Gamma$ and $U$, since these act on the modules $\NN_K(T)$ compatibly with the trace maps.

   \begin{proposition}
    We have an isomorphism of topological modules
    \[ \NN_{F_\infty}(T) \cong \NN_F(T) \htimes_{\cO_F} S_{F_\infty / F}.\]
   \end{proposition}

   \begin{proof}
     Clear by construction.
   \end{proof}

   By construction, $\NN_{F_\infty}(T)$ has $\cO_F$-linear actions of $\Gamma$ and of $U$, which extend to a continuous action of $\Lambda_{\cO_F}(\Gamma \times U)$.
%

   Define $\varphi^* \NN_{F_\infty}(T)$ as the $\AA^+_{F}$-submodule of $\NN_{F_\infty}(T)[q^{-1}]$ generated by $\vp(\NN_{F_\infty}(T))$; this is in fact an $\AA^+_{F} \htimes_{\cO_F} \Lambda_{\cO_F}(U)$-submodule, since $\vp$ acts bijectively on $\Lambda_{\cO_F}(U)$. If $T$ has non-negative Hodge--Tate weights, then we have an inclusion
   \[ \NN_{F_\infty}(T) \into \vp^* \NN_{F_\infty}(T),\]
   with quotient annihilated by $q^h$, for any $h$ such that the Hodge-Tate weights of $T$ lie in $[0, h]$. Note that the map $\vp : \NN_{F_\infty}(T) \to \vp^* \NN_{F_\infty}(T)$ commutes with the action of $G = U \times \Gamma$. Similarly, the maps $\psi$ on $\NN_K(T)[q^{-1}]$ for each $K$ assemble to a map
   \[ \psi : \vp^* \NN_{F_\infty}(T) \to \NN_{F_\infty}(T),\]
   which is a left inverse of $\vp$.

   The following proposition will be important for constructing the regulator map:

   \begin{proposition}
    \label{prop:kernelofpsi}
    We have
    \[ \left( \vp^* \NN_{F_\infty}(T) \right)^{\psi = 0} = \left(\vp^* \NN_F(T)\right)^{\psi = 0} \htimes_{\cO_F} S_{F_\infty / F}.\]
   \end{proposition}

    \begin{proof}
     Choose a basis $n_1, \dots, n_d$ of $\NN_F(T)$ as an $\AA^+_F$-module, and a basis $\Omega$ of $S_{F_\infty / F}$ as a $\Lambda_{\cO_F}(U)$-module. Then any vector $v \in \vp^* \NN_{F_\infty}(T)$ can be uniquely written as
     \[ v = \sum_{i = 0}^{p-1} \sum_{j = 0}^d (1 + \pi)^i \vp(x_{ij}) \cdot (\vp(n_j) \otimes \Omega),\]
     for some $x_{ij} \in \AA^+_{F} \htimes_{\cO_F} \Lambda_{\cO_F}(U)$, since $\{ (1 + \pi)^{i}: 0 \le i \le p-1\}$ is a basis of $\AA^+_{F}$ over $\vp(\AA^+_{F})$.

     Applying $\psi$, we have
     \[ \psi(v) = \sum_{i = 0}^{p-1} \sum_{j = 0}^d \psi\left((1 + \pi)^i\right) x_{ij} \cdot (n_j \otimes \sigma_p^{-1}\Omega),\]
     where $\sigma_p$ is the arithmetic Frobenius element of $\Gal(F_\infty / \Qp)$. The element $\sigma_p^{-1} \Omega$ is also a $\Lambda_{\cO_F}(U)$-generator of $S_{F_\infty / F}$. Moreover, it is well known that $\psi\left((1 + \pi)^i\right)$ is $1$ if $i = 0$ and $0$ if $1 \le i \le p-1$. So we have $\psi(v) = 0$ if and only if $v$ is in the submodule
     \[ \bigoplus_{i = 1}^{p-1} (1 + \pi)^i\vp(\NN_F(T)) \htimes_{\cO_F} S_{F_\infty / F} = \vp^*\NN_F(T)^{\psi = 0} \htimes_{\cO_F} S_{F_\infty/F}.\]
   \end{proof}
%
%
%

  \subsection{Recovering unramified twists}

   Let us pick a finite-rank free $\Zp$-module $M$ equipped with a continuous action of $U$, via a homomorphism $\rho: \Lambda_{\Zp}(U) \to \End_{\Zp}(M)$ as above.

   There is a ``twisting'' map from $M \otimes_{\Zp} \Lambda_{\Zp}(U)$ to itself, defined by $m \otimes [u] \mapsto \rho(u)^{-1} m \otimes [u]$ for $u \in U$. This map intertwines two different actions of $U$: on the left-hand side the action given by
   \[ u \cdot (m \otimes [v]) = m \otimes [u^{-1}v]\]
   and on the right the action given by
   \[ u \cdot (m \otimes [v]) = \rho(u) m \otimes [u^{-1} v].\]
   Taking the completed tensor product with $\OFinf$ (endowed with its natural $U$-action) and passing to $U$-invariants, we obtain a bijection
   \[ i_M: M \otimes_{\Zp} S_{F_\infty/F} \rTo^\cong S_{F_\infty/F} \cdot \left(M \otimes_{\Zp} \OFinf \right)^{U}.\]

   \begin{proposition}
    There is a canonical isomorphism
    \begin{equation}
     \label{eq:unramifiedisom}
     \NN_F(T) \otimes_{\cO_F} \left(\OFinf \otimes_{\Zp} M \right)^{U} \rTo \NN_F(T \otimes_{\Zp} M),
    \end{equation}
    commuting with the actions of $\AA^+_{F}$, $\Gamma$, $\vp$ and $\psi$ (where the latter two elements act on $\OFinf$ as the arithmetic Frobenius and its inverse).
   \end{proposition}

   \begin{proof}
    Wach modules are known to commute with tensor products \cite{berger04}, so it suffices to check that \[ \NN_F(M) = \AA^+_F \otimes_{\cO_F} \left(\OFinf \otimes_{\Zp} M \right)^{U}.\]
    This follows from the fact that there is a canonical embedding of $\OFinf$ into Fontaine's ring $\AA$, hence there is a canonical inclusion
    \[ \left(\OFinf \otimes_{\Zp} M \right)^{U} \subseteq \left(M \otimes_{\Zp} \AA\right)^{H_F} = \DD_F(M).\]
    Since the left-hand side is free of rank $d$ over $\cO_F$, extending scalars to $\AA^+_F$ gives a submodule of $\DD_F(M)$ which is free of rank $d$ over $\AA^+_F$ and clearly satisfies the conditions defining the Wach module $\NN_F(M) \subset \DD_F(M)$.
  \end{proof}

   \begin{remark}
    Suppose (for simplicity) that $F = \Qp$ and $M \cong \Zp$ with $U$ acting via a character $\tau: U \to \Zp^\times$. Since $\left(M \otimes \OFinf \right)^{U}$ is free of rank 1 over $\Zp$, any choice of basis of this space gives a non-canonical isomorphism between $\NN_{\Qp}(T(\tau))$ and $\NN_{\Qp}(T)$ with its $\vp$-action twisted by $\tau(\sigma_p)^{-1}$. However, the isomorphism \eqref{eq:unramifiedisom} \emph{is} canonical and does not depend on any such choice.
   \end{remark}

   \begin{theorem}
    \label{thm:unramifiedtwist}
    There is a canonical isomorphism
    \[ i_{M}: M \otimes_{\Zp} \NN_{F_\infty}(T) \rTo^\cong \NN_{\infty}(M \otimes_{\Zp} T) \]
    which commutes with the actions of $\vp$, $\Gamma$, $\AA^+_{F}$ and $\End_{\cG_F}(M)$, and satisfies
    \[ i_{M}(u \cdot x) = \rho(u)^{-1} u \cdot i_M(x)\]
    for $u \in U$ and $x \in \NN_{F_\infty}(T)$.
   \end{theorem}

   \begin{proof}
    This follows immediately by tensoring the map
    \[ i_M : M \otimes_{\Zp} S_{F_\infty / F} \rTo^\cong S_{F_\infty/F} \cdot \left(M \otimes_{\Zp} \OFinf\right)^{U}\]
    with $\NN_F(T)$, and using the isomorphism \eqref{eq:unramifiedisom}.
   \end{proof}


 \section{The 2-variable p-adic regulator}
  \label{sect:regulator}


  \subsection{A lemma on universal norms}

   Let $F$ be a finite unramified extension of $\Qp$, and let $T$ be a $\Zp$-rep\-re\-sen\-ta\-tion of $\cG_{F}$.

   \begin{definition}
    \label{def:goodcrystalline}
    The representation $T$ is \emph{good crystalline} if $V = T[1/p]$ is crystalline and has non-negative Hodge-Tate weights.
   \end{definition}

   By \cite[Theorem A.3]{berger03}, for any good crystalline $T$ there is a canonical isomorphism
   \[ H^1_{\Iw}(F(\mu_{p^\infty}), T) \rTo^\cong \left( \pi^{-1} \NN_F(T) \right)^{\psi = 1}.\]
   We define a ``residue'' map
   \[ r_{F, V} : H^1_{\Iw}(F(\mu_{p^\infty}), T) \to \Dcris(F, V)\]
   by composing the above isomorphism with the natural map
   \[ \pi^{-1} \NN_F(T) \to \frac{\pi^{-1} \NN_F(V)}{\NN_F(V)} \cong \frac{\NN_F(V)}{\pi\NN_F(V)} \cong \Dcris(F, V).\]
   As is shown in the proof of \cite[Theorem A.3]{berger03}, the image of the map $r_{F, V}$ is contained in $\Dcris(F, V)^{\vp = 1}$; in particular, if the latter space is zero, then $H^1_{\Iw}(F(\mu_{p^\infty}), T) \cong \NN_F(T)^{\psi = 1}$.

   We now consider the behaviour of these maps in unramified towers. Let $F_\infty$ be an infinite unramified $p$-adic Lie extension of $F$, so we may write $F_\infty = \bigcup_n F_n$ where $F_0 / F$ is a finite extension and $F_n$ is the unramified extension of $F_0$ of degree $p^n$. As we have seen above, $\Dcris(F_n, V) \cong \Dcris(V) \otimes_{F} F_n$. Let us formally write $\Dcris(F_\infty, V) = F_\infty \otimes_{F} \Dcris(F, V)$.

   \begin{proposition}
    There is an $n_0$ (depending on $V$) such that
    \[ \Dcris(F_\infty, V)^{\vp = 1} = \Dcris(F_n, V)^{\vp = 1}.\]
   \end{proposition}

   \begin{proof}
    Since the spaces $\Dcris(F_n, V)^{\vp = 1}$ are an increasing sequence of finite-dimensional $\Qp$-vector spaces, it suffices to show that their union $\Dcris(F_\infty, V)^{\vp = 1}$ is finite-dimensional over $\Qp$. This follows from the fact that $F_\infty$ is a field, and $\vp$ acts on $F_\infty$ as the arithmetic Frobenius $\sigma_p$, so $(F_\infty)^{\vp = 1} = \Qp$. Thus
    \[\dim_{\Qp} \left(F_\infty \otimes_F \Dcris(V)\right)^{\sigma_p \otimes \vp = 1} \le \dim_{\Qp} V,\]
    by Propositions 1.4.2(i) and 1.6.1 of \cite{fontaine94b}.
   \end{proof}

   \begin{proposition}
    Let $\Dcris(T)$ be the $\Zp$-lattice in $\Dcris(V)$ which is the image of $\NN_F(T)$. If $m \ge n \ge n_0$, $x \in H^1_{\Iw}(F_{m}(\mu_{p^\infty}), T)$, and $y = \operatorname{cores}_{F_m/F_{n}}(x) \in H^1_{\Iw}(F_{n}(\mu_{p^\infty}), T)$, then we have
    \[ r_{F_{n}, V}(y) \in p^{m-n} \cO_{F_n} \otimes_{\cO_F} \Dcris(T).\]
   \end{proposition}

   \begin{proof}
    This follows from the fact that for any $n \ge 0$, we have a commutative diagram
    \[
     \begin{diagram}
      H^1_{\Iw}(F_{n+1}(\mu_{p^\infty}), T) & \rTo^{r_{F_{n+1}, V}} & \Dcris(F_{n+1}, V)^{\vp = 1}\\
      \dTo^{\operatorname{cores}_{F_{n+1}/F_n}} & & \dTo^{\operatorname{Tr}_{F_{n+1}/F_n}} \\
      H^1_{\Iw}(F_n(\mu_{p^\infty}), T) & \rTo^{r_{F_{n}, V}} & \Dcris(F_n, V)^{\vp = 1}.
     \end{diagram}
    \]
    If $n \ge n_0$, then the trace map on the right-hand side is simply multiplication by $[F_{n+1} : F_n] = p$.
   \end{proof}

   \begin{theorem}
    \label{thm:unramunivnorms}
    Let $F_\infty$ be an infinite unramified $p$-adic Lie extension of $F$, and let $x \in H^1_{\Iw}(F_\infty(\mu_{p^\infty}), T)$. Then for any $n \ge 0$, the image $y$ of $x$ in $H^1_{\Iw}(F(\mu_{p^\infty}), T) \cong \left(\pi^{-1}\NN_F(T)\right)^{\psi = 1}$ is contained in $\NN_F(T)^{\psi = 1}$.
   \end{theorem}

   \begin{proof}
    This follows immediately from the preceding proposition, since $r_{F_{n}, V}(y)$ must be divisible by arbitrarily large powers of $p$ and hence is zero.
   \end{proof}


  \subsection{The regulator map}
   \label{sect:2variableregulator}

   For the rest of Section \ref{sect:regulator}, we assume that $T$ is a good crystalline representation of $\cG_{F}$, for $F$ a finite unramified extension of $\Qp$, and we let $F_\infty$ be any unramified $p$-adic Lie extension of $F$ with Galois group $U$ as before. We define $K_\infty = F_\infty(\mu_{p^\infty})$, and $G = \Gal(K_\infty / F) \cong U \times \Gamma$.

   \begin{proposition}
    \label{prop:yagercohomology2}
    We have a canonical isomorphism
    \[ H^1_{\Iw}(K_\infty ,T) \cong \left(\pi^{-1}\NN_{F_\infty}(T)\right)^{\psi = 1}.\]
    If either $F_\infty / F$ is infinite, or $T$ has no quotient isomorphic to the trivial representation, then we have
    \[  H^1_{\Iw}(K_\infty, T) \cong \NN_{F_\infty}(T)^{\psi = 1}.\]
   \end{proposition}

   \begin{proof}
    If $F_\infty / F$ is a finite extension, we may assume $F_\infty = F$, and this is \cite[Theorem A.1]{berger03}.

    If $F_\infty / F$ is an infinite extension, then we note that for each finite subextension $K / F$ contained in $F_\infty$ we have an isomorphism
    \[ H^1_{\Iw}(K(\mu_{p^\infty}),T) \cong \left(\pi^{-1}\NN_{K}(T)\right)^{\psi=1},\]
    and if $L / K$ are two such fields, then the corestriction map
    \[ H^1_{\Iw}(L(\mu_{p^\infty}),T)\rTo H^1_{\Iw}(K(\mu_{p^\infty}),T)\]
    corresponds to the maps
    \[ \pi^{-1}\NN_{L}(T)\rTo \pi^{-1}\NN_{K}(T)\]
    induced from the trace map $\cO_L \to \cO_K$. By Theorem \ref{thm:unramunivnorms}, we have an isomorphism
    \[
     H^1_{\Iw}(K_\infty, T) = \varprojlim_K \left( \pi^{-1} \NN_{K}(T)\right)^{\psi=1} \cong \varprojlim_K \NN_{K}(T)^{\psi=1}=
     \NN_{F_\infty}(T)^{\psi=1},
    \]
    which finishes the proof.
   \end{proof}

   As shown in \cite[Proposition 2.11]{leiloefflerzerbes11}, we have a $\Lambda_{\cO_F}(\Gamma)$-equivariant embedding
   \[ \big(\vp^*\NN_F(T)\big)^{\psi=0}\subset \cH_{F}(\Gamma) \otimes_{F} \Dcris(V),\]
   which is continous with respect to the weak topology on $\vp^* \NN_F(T)^{\psi=0}$ and the usual Fr\'echet topology on $\cH_{F}(\Gamma)$. Moreover, we have a continuous injection
   \[ S_{F_\infty / F} \into \Lambda_{\OFinf}(U) \into \cH_{\Finf}(U).\]
   Tensoring these together we obtain a continuous, $\Lambda_{\cO_F}(G)$-linear map
   \begin{multline*}
    \big(\vp^*\NN_F(T)\big)^{\psi=0} \htimes_{\cO_F} S_{F_\infty / F} \into \cH_{\Finf}(U) \htimes_{\cO_F} \cH_F(\Gamma) \otimes_{F} \Dcris(V) \\ = \cH_{\Finf}(G) \otimes_{F} \Dcris(V).
   \end{multline*}

   \begin{definition}
    \label{def:2variableregulator}
    We define the $p$-adic regulator
    \[ \cL^{G}_V : H^1_{\Iw}(K_\infty, T) \rTo \cH_{\Finf}(G) \otimes_{F} \Dcris(V)\]
    to be the composite map
    \begin{align*}
     H^1_{\Iw}(K_\infty, T) & \rTo^\cong \NN_{F_\infty}(T)^{\psi=1} \cong \left(\NN_F(T) \htimes_{\cO_F} S_\infty\right)^{\psi = 1}\\
     & \rTo^{1- \vp} \big(\vp^*\NN_F(T)\big)^{\psi=0} \htimes_{\cO_F} S_\infty \\
     & \rTo \cH_{\Finf}(G) \otimes_{F} \Dcris(V).
    \end{align*}
    Here, we use that $\vp^* \NN_{F_\infty}(T)^{\psi = 0} \cong \vp^*\NN_F(T)^{\psi=0} \htimes_{\Zp} S_\infty$ by Proposition \ref{prop:kernelofpsi}.
   \end{definition}

   By construction, $\cL^{G}_V$ is a morphism of $\Lambda_{\cO_F}(G)$-modules. As suggested by the notation, we will usually invert $p$ and regard $\cL^{G}_V$ as a map on $H^1_{\Iw}(K_\infty, V)$, associating to each compatible system of cohomology classes in $H^1_{\Iw}(K_\infty, V)$ a distribution on $G$ with values in $\Finf \otimes_{F} \Dcris(V)$.

    We can summarise the properties of the map we have constructed by the following theorem:

   \begin{theorem}
    \label{thm:localregulator}
    Let $F$ be a finite unramified extension of $\Qp$, and $K_\infty$ a $p$-adic Lie extension of $F$ with Galois group $G$ such that
    \[ F(\mu_{p^\infty}) \subseteq K_\infty \subset F \cdot \Qp^{\mathrm{ab}}.\]

    Let $T$ be a crystalline representation of $\cG_{F}$ with non-negative Hodge--Tate weights, and assume that either $K_\infty / F(\mu_{p^\infty})$ is infinite, or $T$ has no quotient isomorphic to the trivial representation.

    Then there exists a morphism of $\Lambda_{\cO_F}(G)$-modules
    \[ \mathcal{L}_V^G : H^1_{\Iw}(K_\infty, T) \rTo \cH_{\Finf}(G) \otimes_{F} \Dcris(V),\]
    where $F_\infty$ is the maximal unramified subfield of $K_\infty$, such that:
    \begin{enumerate}
     \item for any finite unramified extension $K / F$ contained in $K_\infty$, we have a commutative diagram
     \[
      \begin{diagram}
       H^1_{\Iw}(K_\infty / \Qp, V) &&\rTo^{\mathcal{L}_V^G} &&\cH_{\Finf}(G) \otimes_{F} \Dcris(V)\\
       \dTo & & & &\dOnto \\
       H^1_{\Iw}(K(\mu_{p^\infty}), V )& \rTo^{\mathcal{L}_V^{G'}} & \cH_{K}(G') \otimes_{F} \Dcris(V) & \rInto & \cH_{\Finf}(G') \otimes_{F} \Dcris(V).
      \end{diagram}
     \]
     Here $G' = \Gal(K(\mu_{p^\infty}) / \Qp)$, the right-hand vertical arrow is the map on distributions corresponding to the projection $G \onto G'$, and the map $\mathcal{L}_V^{G'}$ is defined by
     \[ \mathcal{L}_V^{G'} = \sum_{\sigma \in \Gal(K / F)} [\sigma] \cdot \mathcal{L}^{\Gamma}_{K, V}(\sigma^{-1} \circ x),\]
     where $\cL^\Gamma_{K, V}$ is the Perrin-Riou regulator map for $K(\mu_{p^\infty}) / K$.
     \item For any $x \in H^1_{\Iw}(F_\infty(\mu_{p^\infty}) / \Qp, V)$ and any character $\eta$ of $\Gamma$, the distribution $\pr^{\eta}(\cL_{G, V}(x))$ on $U$, which is defined by twisting by $\eta$ and pushing forward along the projection to $U$, is bounded.
    \end{enumerate}
    Moreover, the conditions (1) and (2) above uniquely determine the morphism $\mathcal{L}_V^G$.
   \end{theorem}

   \begin{proof}
    Let us show first that the map $\mathcal{L}_{V}^G$ defined above satisfies (1) and (2). Let $T$ be a choice of lattice in $V$.

    Let $K$ be any finite unramified extension of $F$ contained in $F_\infty$. Then the diagram
    \[
     \begin{diagram}
      \NN_{F_\infty}(T)^{\psi = 1} & \rTo^{1-\vp} & \vp^*\NN_{F_\infty}(T)^{\psi = 0}\\
      \dTo & & \dTo\\
      \NN_K(T)^{\psi = 1} & \rTo^{1-\vp} & \vp^*\NN_{K}(T)^{\psi = 0}
     \end{diagram}
    \]
    evidently commutes; and we also have a commutative diagram
    \[
     \begin{diagram}
      \vp^* \NN_F(T)^{\psi = 0} \htimes_{\cO_F} S_{F_\infty/F} & \rTo^{i_{F_\infty / F}} & \cH_{\Finf}(G) \otimes_{F} \Dcris(V)  \\
      \dTo & & \dTo\\
      \vp^* \NN_K(T)^{\psi = 0} \otimes_{\Zp} S_{K/F} & \rTo^{i_{K/F}} & \cH_{\Finf}(G') \otimes_{F} \Dcris(V),
     \end{diagram}
    \]
   where the arrows $i_{F_\infty / F}$ and $i_{K/F}$ are induced by the inclusions $S_{F_\infty/F} \into \cH_{\Finf}(U)$ and $S_{K/F} \into \cO_K[U'] \into \Finf[U']$, where $U' = \Gal(K/F)$, and the right vertical arrow is the one arising from the projection $G \to G'$. If we combine the two diagrams using the identification $\NN_K(T) \cong \NN_F(T) \otimes_F S_{K/F}$ and similarly for $F_\infty$, the composite of the maps on the top row is the definition of $\mathcal{L}_{V}^G$, and the composite of the arrows on the bottom row is the map $\mathcal{L}_V^{G'}$. The commutativity of these diagrams therefore proves (1).

     Property (2) is clear, since the image of $\Lambda_{\Finf}(U)$ in $\cH_{\Finf}(U)$ is exactly the bounded distributions.

    We now show that these properties characterise $\mathcal{L}^G_{V}$ uniquely. It suffices to show that (1) and (2) determine the value of $\mathcal{L}^G_{V}(x)$ at any character of $G$. Such a character has the form $\eta \varpi$ where $\eta$ is a character of $\Gamma$ and $\varpi$ is a character of $U$. Property (1) uniquely determines the value at $\eta \times \varpi$ if $\varpi$ has finite order, and property (2) implies that for each fixed $\eta$, the function $\varpi \mapsto \mathcal{L}^G_{V}(x)(\eta \times \varpi)$ is a bounded analytic function on the rigid space parametrising characters of $U$, and hence is determined uniquely by its values on finite-order $\varpi$'s.
   \end{proof}

   We now record some properties of the map $\cL^G_{V}$.

   \begin{proposition}
    \label{prop:orderofdistributions}
    Let $W \subseteq \Dcris(V)$ be a $\vp$-invariant $F$-subspace such that all eigenvalues of $\vp$ on the quotient $Q = \Dcris(V) / W$ have $p$-adic valuation $\ge -h$ (where we normalise the $p$-adic valuation on $\Qpb$ such that $v_p(p) = 1$).

    Then for any $x \in H^1_{\Iw}(K_\infty, V)$, the image of $x$ under
    \[ H^1_{\Iw}(K_\infty, V) \rTo^{\cL^G_V} \cH_{\Finf}(G) \otimes_{F} \Dcris(V) \to \cH_{\Finf}(G) \otimes_{F} Q\]
    lies in $D^{(0,h)}(G, \Finf) \otimes Q$, where $D^{(0,h)}(G, \Finf)$ is the space of $\Finf$-valued distributions of order $(0, h)$ with respect to the subgroups $(U, \Gamma)$.
   \end{proposition}

   \begin{proof}
    This is immediate from the definition of the 2-variable regulator map and the corresponding statement for the 1-variable regulator, which is well known.
    \end{proof}
%

   \begin{proposition}
    \label{prop:galoisequivariance}
    If $u \in U$ and $\tilde u$ is the unique lifting of $u$ to $G$ acting trivially on $F(\mu_{p^\infty})$, then for any $x \in H^1_{\Iw}(K_\infty, V)$ we have
    \[\cL^G_{V}(x)^u = [\tilde u] \cdot \cL^G_{V}(x).\]
   \end{proposition}

   \begin{proposition}
    If $m_1, \dots, m_d$ are a $\Lambda_{\cO_F}(\Gamma)$-basis of $\vp^* \NN_F(T)^{\psi = 0}$, and $\omega$ a $\Lambda_{\cO_F}(U)$-basis of $S_\infty$, then the image of the $p$-adic regulator is contained in the $\Lambda_{\cO_F}(G)$-span of the vectors
    \[ \left(i_\infty(m_j \htimes \omega)\right)_{j = 1, \dots, d}.\]
   \end{proposition}

   \begin{proposition}
    \label{prop:injectivereg}
    If $F_\infty / F$ is infinite, the regulator map $\cL^{G}_V$ is injective.
   \end{proposition}

   \begin{proof}
    As before, let us identify $F_\infty$ with the unramified $\Zp$-extension of a finite extension $F_0 / F$. Let $\pr_n:\NN_{F_\infty}(T) \rightarrow \NN_{F_n}(T)$ be the projection map; for $x\in\NN_{F_\infty}(T)$, we have $\vp(x) = x$ if and only if $\pr_n(x) \in \NN_{F_n}(T)^{\vp = 1}$ for all $n$. However,
    \[ \NN_{F_n}(T)^{\vp = 1} \subset \DD_{F_n}(T)^{\vp=1} = T^{H_{F_n}}.\]
    As $T$ is a finitely generated $\Zp$-module, there must be some $m$ such that $T^{H_{F_n}} = T^{H_{F_m}}$ for all $n \ge m$. However, for $n \ge m$ the projection map $T^{H_{F_{n+1}}} \to T^{H_{F_{n}}}$ is multiplication by $p$; so $\pr_m(x)$ is divisible by arbitrarily high powers of $p$ and is thus zero. Hence $x = 0$.
   \end{proof}

   The next statement requires some extra notation. Let $\varpi$ be a continuous character $U \to \cO_E^\times$, where $E$ is some finite extension of $\Qp$. Then there is an obvious isomorphism
   \begin{equation}
    \label{eq:iwacohotwist}
    H^1_{\Iw}(K_\infty, T(\varpi)) \cong H^1_{\Iw}(K_\infty, T)(\varpi).
   \end{equation}
   Moreover, via the isomorphism $V \otimes_{\Qp} \Bcris \cong \Dcris(V) \otimes_{F} \Bcris$, we can regard the space
   \[ \Dcris(V(\varpi)) = \left( E \otimes_{\Qp} V \otimes_{\Qp} \Bcris\right)^{\cG_{\Qp} = \varpi^{-1}}\]
   as a subspace of $E \otimes_{\Qp} \Dcris(V) \otimes_{F} \Bcris$.
   Since the natural inclusion $\Finf \into \Bcris$ induces an injection
   \[ (E \otimes_{\Qp} K_\infty)^{U = \varpi^{-1}} \into (E \otimes_{\Qp} \Bcris)^{\cG_{\Qp} = \varpi^{-1}}\]
   which must be an isomorphism (as the right-hand side must have $E$-dimension $\le 1$), we have a canonical isomorphism
   \[ \Dcris(V(\varpi))  = \Dcris(V) \otimes_{F} \left(E \otimes_{\Qp} K_\infty\right)^{U = \varpi^{-1}}.\]
   In particular, there is a canonical isomorphism
   \[ \Finf \otimes_{F} \Dcris(V(\varpi)) \cong E \otimes_{\Qp} \Finf \otimes_{F} \Dcris(V).\]

   We also have a canonical map
   \[ \Tw_{\varpi^{-1}} : E \otimes_{\Qp} \cH_{\Finf}(G) \to E \otimes_{\Qp} \cH_{\Finf}(G)\]
   which on group elements corresponds to the map $g \mapsto \varpi(g)^{-1} g$. Tensoring with the canonical isomorphism above, we obtain a map (which we also denote by $\Tw_{\varpi^{-1}}$)
   \[ E \otimes_{\Qp} \cH_{\Finf}(G) \otimes_{F} \Dcris(V) \rTo \cH_{\Finf}(G) \otimes_{F} \Dcris(V(\varpi)).\]

   \begin{proposition}
    \label{prop:twisting}
    With the identifications described above, the regulator $\cL^G_V$ is invariant under unramified twists: there is a commutative diagram
    \begin{diagram}
     \cO_E \otimes H^1_{\Iw}(K_\infty, T) & \rTo^{\cL_V^G} & E \otimes_{\Qp} \cH_{\Finf}(G) \otimes_{F} \Dcris(V) \\
     \dTo^\cong & & \dTo^{\Tw_{\varpi^{-1}}}\\
     H^1_{\Iw}(K_\infty, T(\varpi)) & \rTo^{\cL_{V(\varpi)}^{G}} & \cH_{\Finf}(G) \otimes_{F} \Dcris(V(\varpi))
    \end{diagram}
   \end{proposition}

   \begin{proof}
    By  \eqref{eq:iwacohotwist}, we have canonical isomorphisms $H^1_{\Iw}(K_\infty, T) \otimes_{\Zp} \cO_E \cong \NN_{F_\infty}(T)^{\psi = 1} \otimes_{\Zp} \cO_E$, and $H^1_{\Iw}(K_\infty, T(\varpi)) \cong \NN_{F_\infty}(T(\varpi))^{\psi = 1}$. We can therefore rewrite the above diagram to obtain the following:
    \begin{diagram}
     \NN_{F_\infty}(T)^{\psi = 1} \otimes_{\Zp} \cO_E & \rTo^{1-\vp} & \NN_{F_\infty}(T)^{\psi = 0} \otimes_{\Zp} \cO_E & \rTo & \cH_{\Finf}(G) \otimes_{\Qp} \Dcris(V) \otimes_{\Qp} E \\
     \dTo & & \dTo & & \dTo^{\Tw_{\varpi^{-1}}}\\
     \NN_{F_\infty}(T(\varpi))^{\psi = 1} & \rTo^{1-\vp} & \NN_{F_\infty}(T(\varpi))^{\psi = 0} & \rTo & \cH_{\Finf}(G) \otimes_{\Qp} \Dcris(V(\varpi)).
    \end{diagram}

    Here the left and middle vertical maps are obtained by restriction from that of Theorem \ref{thm:unramifiedtwist}, taking $\tau = \varpi^{-1}$; as noted above, this isomorphism commutes with $\vp$ and $\psi$.

    The commutativity of the left square is clear. Moreover, the isomorphisms
    \[ \NN_F(T(\varpi)) \cong \NN_F(T) \otimes_{\Zp} \left( \cO_E \otimes_{\Zp} \OFinf\right)^{U = \varpi^{-1}}\]
    and
    \[ \Dcris(V(\varpi)) \cong \Dcris(V) \otimes_{\Zp} \left( \cO_E \otimes_{\Zp} \OFinf \right)^{U = \varpi^{-1}}\]
    are compatible (since the first is given by multiplication in $\AA$, the second in $\Bcris$, and the inclusion of $\OFinf$ in $\Bcris$ factors through the natural maps $\AA^+ \into \tilde{\AA}^+ \into \AA_{\cris}$). Hence the commutativity of the right square follows, as the twisting maps $\Lambda_{\OFinf}(U) \to \Lambda_{\OFinf}(U)$ and $\cH_{\Finf}(U) \to \cH_{\Finf}(U)$ are evidently compatible.
%
%

   \end{proof}


  \subsection{An explicit formula for the values of the regulator}
   \label{sect:explicitformula}

   In this section, we use the results from the previous section to give a direct interpretation of the value of the regulator map $\cL^G_V$ at any de Rham character of $G$, relating these to the values of the Bloch-Kato exponential maps for $V$ and its twists. In this section we assume (for simplicity) that $F = \Qp$.

   As above, let $\varpi$ be a continuous character of $U$ with values in $\cO_E$, for some finite extension $E / \Qp$. Combining Proposition \ref{prop:twisting} with the defining property of $\cL^G_{V(\varpi)}$ in Theorem \ref{thm:localregulator}, we have:

   \begin{theorem}
    The following diagram commutes:
    \[
     \begin{diagram}
      H^1_{\Iw}(K_\infty, V) & \rTo^{\cL_{V}^G} & \cH_{\Finf}(G)^{\circ} \otimes_{\Qp} \Dcris(V) \\
      \dTo_{\pr_{\Iw}^\varpi} & & \dTo_{\pr^\varpi_{\cris}}\\
      H^1_{\Iw}(\Qp(\mu_{p^\infty}), V(\varpi)) & \rTo^{\cL^\Gamma_{V(\varpi)}} & \cH_{\Qp}(\Gamma) \otimes_{\Qp} \Dcris(V(\varpi)).
     \end{diagram}
    \]
   \end{theorem}

   Here $\cH_{\Finf}(G)^\circ$ denotes the subspace of $\cH_{\Finf}(G)$ satisfying the Galois-equivariance property of Proposition \ref{prop:galoisequivariance}. The map $\pr_{\Iw}^\varpi$ is the composite of the isomorphism \eqref{eq:iwacohotwist} with the corestriction map; the right-hand vertical map is the composite of $\Tw_{\varpi^{-1}}$ with push-forward to $\Gamma$. (Hence both vertical maps are $U$-equivariant, if we let $U$ act on the bottom row by $\varpi^{-1}$.)

   We now apply the results of \S \ref{appendix:cyclo} to each unramified twist $V(\varpi)$ of $V$ to determine exactly the values of $\cL_V^G$ at any character of $G$ which is Hodge--Tate, in terms of the dual exponential and logarithm maps (cf.~\S\ref{sect:crystallinereps} above).

   \begin{definition}
    Let $\omega$ be any continuous character of $G$ with values in some finite extension $E / \Qp$. For $x\in H^1_{\Iw}(K_\infty,V)$, we write $x_{\omega, 0}$ for the image of $x$ in $H^1(\Qp, V(\omega^{-1}))$.
   \end{definition}

    We can now apply Theorem \ref{thm:explicitformulacyclo} to obtain the following formulae for the values of $\cL^G_V$:

   \begin{theorem}
    \label{thm:explicitformula}
    Let $x\in H^1_{\Iw}(K_\infty,V)$. Let $j$ be the Hodge--Tate weight of $\omega$, and $n$ its conductor. If $n = 0$, suppose that $\Dcris(V(\omega^{-1}))^{\vp = p^{-1}} = 0$. Then we have
    \begin{multline*}
     \cL^G_{V}(x)(\omega) =
     \Gamma^*(1+j) \cdot \ve(\omega^{-1}) \cdot \frac{\Phi^n P(\omega^{-1}, \Phi)}{P(\omega, p^{-1} \Phi^{-1})} \\
     \times \begin{cases}
       \exp^*_{V(\omega^{-1})^*(1)}(x_{\omega,0}) \otimes t^{-j} e_j &  \text{if $j \ge 0$,}\\
       \log_{\Qp, V(\omega^{-1})}(x_{\omega,0}) \otimes t^{-j} e_j & \text{if $j \le -1$,}
     \end{cases}
    \end{multline*}
    where the notation is as follows:
    \begin{itemize}
     \item $\Gamma^*(1 + j)$ is the leading term of the Taylor expansion of the Gamma function at $1 + j$,
     \[ \Gamma^*(1 + j) =
      \begin{cases}
       j! & \text{if $j \ge 0$,}\\
       \frac{(-1)^{-j-1}}{(-j-1)!} & \text{if $j \le -1$.}
      \end{cases}
     \]
     \item $P_\omega$ and $\ve(\omega)$ are the $L$ and $\ve$-factors of the Weil--Deligne representation $\DD_{\mathrm{pst}}(\omega)$ (see \S \ref{sect:epsfactors} above).
     \item $\Phi$ denotes the operator on $\Dcris(V) \otimes_{\Qp} \Finf$ which is obtained by extending the Frobenius of $\Dcris(V)$ to act trivially on $\Finf$ (rather than as the usual Frobenius on $\Finf$).
    \end{itemize}
   \end{theorem}

   \begin{remark}
    To define $\cL^G_{V}$ we made a choice of compatible system of $p$-power roots of unity $\zeta$; but the dependence of $\cL^G_V$ on $\zeta$ is clear from the formula of Theorem \ref{thm:explicitformula}. If we temporarily write $\cL^{G}_{V}(x, \zeta)$ for the regulator using the roots of unity $\zeta$, then for any $\gamma \in \Gamma$ we have
    \[ \cL^G_V(x, \gamma \zeta)(\omega) = \omega(\tilde \gamma)^{-1} \cL^G_V(x, \zeta)(\omega),\]
    where $\tilde \gamma$ is the unique lifting of $\gamma$ to the inertia subgroup of $G$.
   \end{remark}

  \subsection{A local reciprocity formula}
   \label{sect:localrecip}

   Our final local result will be an analogue of Perrin-Riou's local reciprocity formula, relating the maps $\cL^G_V$ and $\cL^G_{V^*(1)}$. The cyclotomic version of this formula, conjecture $\operatorname{Rec}(V)$ in \cite{perrinriou94}, was originally formulated in terms of Perrin-Riou's exponential map $\Omega_{V, h}$, and proved independently by Colmez \cite{colmez98} and Benois \cite{benois00}. In Appendix \ref{appendix:cyclo} below we formulate and prove a version using the map $\cL^\Gamma_V$ instead.

   Here, as in Appendix \ref{appendix:cyclo}, it will be convenient to us to extend the definition of the regulator map to representations which are crystalline, but which may have some negative Hodge--Tate weights. To do this, we note that if $V$ is good crystalline, then for any $k \ge 0$ we have
   \[ \Tw_{\chi}\left( \cL^G_{V(1)}(x \otimes e_{1})\right) = \ell_{-1}\left(\cL^G_V(x)\right) \otimes t^{-1} e_1.\]
   So for arbitrary $V$, and any $j \gg 0$ such that $V(j)$ is good crystalline, we may \emph{define} $\cL^G_{V}$ by the formula
   \[ \cL^G_V(x) = (\ell_{-1} \circ \dots \circ \ell_{-j})^{-1}\left( \Tw_{\chi^j}\left(\cL^G_{V(j)}(x \otimes e_{j})\right)\right) \otimes t^j e_{-j}\]
   and this does not depend on the choice of $j$; this then takes values in the fraction field of $\cH_{\Finf}(G)$.

   \begin{theorem}
    For any crystalline representation $V$ and any classes $x \in H^1_{\Iw}(K_\infty, V)$ and $y \in H^1_{\Iw}(K_\infty, V^*(1))$, we have
    \[ \langle \cL_V^G(x), \cL_{V^*(1)}^G(y) \rangle_{\mathrm{cris}, V} = -\sigma_{-1} \cdot \ell_0 \cdot \langle x, y \rangle_{K_\infty, V},\]
    where $\sigma_{-1}$ denotes the unique element of the inertia subgroup of $G$ such that $\chi(\sigma_{-1}) = -1$.
   \end{theorem}

   \begin{proof}
    Using Lemma \ref{lemma:PRpairingtwist} and Proposition \ref{prop:twisting} for each unramified character $\tau$ of $G$ reduces this immediately to the corresponding statement for the cyclotomic regulator maps $\cL^\Gamma_{V(\tau)}$, which is Theorem \ref{thm:cycloreciprocity}.
   \end{proof}

 \section{Regulators for extensions of number fields}
  \label{sect:globalregulator}
  \label{sect:semilocalregulator}
  \newcommand{\Kpb}{\overline{K}_{\fp}}

  In this section, we show how to define an extension of the regulator map in the context of certain $p$-adic Lie extensions of number fields. This section draws heavily on the cyclotomic case studied by Perrin-Riou in \cite{perrinriou94}; see also \cite{iovitapollack06} for the case of more general $\Zp$-extensions of number fields.

  Let $K$ be a number field, $p$ a (rational) prime, and $\fp$ a prime of $K$ above $p$. We choose a prime $\fP$ of $\overline{K}$ above $\fp$.

  \subsection{Semilocal cohomology}

   Let $T$ be a finitely generated $\Zp$-module with a continuous action of $\cG_K$. For each finite extension $L$ of $K$, the set of primes $\fq$ of $L$ above $\fp$ is finite, and for each $i$ we may define the semilocal cohomology group
   \[ Z^i_{\fp}(L, T) = \bigoplus_{\fq \mid \fp} H^i(L_\fq, T).\]
   If $L / K$ is Galois, with Galois group $G$, then we have a canonical isomorphism
   \begin{equation}
    \label{eq:semilocalcoho}
    Z^i_{\fp}(L, T) \cong \Zp[G] \otimes_{\Zp[G_\fP]} H^i(L_\fP, T),
   \end{equation}
   where $G_\fP$ is the decomposition group of $\fP$ in $G$. In particular, it has an action of $\Zp[G]$, and it is easy to see that the localization map
   \[ \loc_{\fp} = \bigoplus_{\fq \mid \fp} \loc_\fq : H^i(L, T) \to Z^i_\fp(L, T)\]
   is $G$-equivariant.

   If now $K_\infty / K$ is a $p$-adic Lie extension of number fields with Galois group $G$, we may define semilocal Iwasawa cohomology groups
   \[ Z^i_{\Iw, \fp}(K_\infty, T) = \varprojlim_{K'} Z^i_{\fp}(K', T),\]
   where the inverse limit is over finite Galois extensions $K' / K$ contained in $K_\infty$. The isomorphisms \eqref{eq:semilocalcoho} for each finite subextension imply that
   \begin{equation}
    \label{eq:semilocaliwasawacoho}
    Z^i_{\Iw, \fp}(K_\infty, T) = \Lambda_{\Zp}(G) \otimes_{\Lambda_{\Zp}(G_\fP)} H^i_{\Iw}(K_{\infty, \fP}, T).
   \end{equation}

   \begin{theorem}
    \label{thm:semilocalreg}
    Let $K_\infty / K$ be any $p$-adic Lie extension of number fields with Galois group $G$, $\fp$ a prime of $K$ above $p$, and $\fP$ a prime of $\overline{K}$ above $\fp$, such that
    \begin{itemize}
     \item $K_\fp$ is unramified over $\Qp$,
     \item the completion $K_{\infty, \fP}$ is of the form $F_\infty(\mu_{p^\infty})$, for $F_\infty$ an infinite unramified extension of $K_{\fp}$.
    \end{itemize}
    Then there is a unique homomorphism of $\Lambda_{\Zp}(G)$-modules
    \[ \cL^G_{V} : Z^1_{\Iw, \fp}(K_\infty, V) \to \cH_{\Finf}(G) \otimes_{\Qp} \Dcris(K_\fp, V), \]
    where $\Finf$ is the $p$-adic completion of the maximal unramified subfield of $K_{\infty, \fP}$, whose restriction to $H^1_{\Iw}(K_{\infty, \fP}, V)$ is the local regulator map $\cL^{G_\fP}_V$.
   \end{theorem}

   \begin{proof}
    Immediate by tensoring the local regulator $\cL^{G_\fP}_V$ with $\Lambda_{\Zp}(G)$, using equation \eqref{eq:semilocaliwasawacoho}.
   \end{proof}

   \begin{remark}
    Note that if $\fp$ is finitely decomposed in $K_\infty$, so $[G : G_{\fP}]$ is finite, one can describe $\cL^G_{V}$ as a direct sum of local regulators:
    \[ \cL^G_{V}(x) = \bigoplus_{\sigma \in G / G_{\fP}} [\sigma] \cdot \cL^{G_{\fP}}_{V}(\loc_{\fp} \sigma^{-1}(x)).\]
    However, the construction also applies when $\fp$ is infinitely decomposed. Thus, for instance, if $d > 1$ and $K$ is a CM field of degree $2d$ in which $p$ splits completely, then one can take $K_\infty$ to be the $(d + 1)$-dimensional abelian $p$-adic Lie extension given by the ray class field $K(p^\infty)$.
   \end{remark}

   \begin{remark}
    One can use the regulator maps to construct Coleman maps and restricted Selmer groups of $V$ over $K_\infty$, in the spirit of the constructions in \cite{leiloefflerzerbes10} for the cyclotomic extension.
   \end{remark}


  \subsection{The module of p-adic L-functions}
   \label{sect:defmoduleofLfunctions}

   We now assume that the number field $K$ is totally complex and Galois over $\QQ$, and that $p$ splits completely in $K$, $(p)=\fp_1\dots \fp_e$. For each of these primes, fix an embedding of $\Qb$ into $\overline{K_{\fp_i}}$.
   Let $T$ be a $\Zp$-representation of $\cG_K$, and let $V=T[p^{-1}]$.

   \begin{assumption}
    \label{assumption:crystalline}
    For all $1\leq i\leq e$, the restriction of $V$ to $\cG_{K_{\fp_i}}$ is good crystalline.
   \end{assumption}

   Let $S$ be the finite set of primes of $K$ containing  all the primes above $p$, all the archimedean places and all the places whose inertia group acts non-trivially on $T$. Denote by $K^S$ the maximal extension of $K$ unramified outside $S$. Let $K_\infty$ be a $p$-adic Lie extension of $K$ contained in $K^S$ which is Galois over $\QQ$ and satisfies the conditions of Theorem \ref{thm:semilocalreg} for each of the primes $\fp_1, \dots, \fp_e$.

   \begin{definition}
    Define $H^1_{\Iw,S}(K_\infty,T)=\varprojlim H^1(\Gal(K^S\slash K_n),T)$, where $\{K_n\}$ is a sequence of finite extensions of $K$ such that $K_\infty=\bigcup K_n$. We also let
    \[ H^1_{\Iw,S}(K_\infty,V)=H^1_{\Iw,S}(K_\infty,T)\otimes_{\Zp}\Qp.\]
   \end{definition}

   \begin{assumption}
    \label{assumption:eltsorderp}
    The Galois group $G = \Gal(K_\infty / K)$ has no element of order $p$.
   \end{assumption}

   \begin{remark}
    Examples of $p$-adic Lie extensions satisfying the above hypotheses occur naturally in the context of class field theory; for instance, if $K$ is a CM field in which $p$ splits, and $K_\infty$ the ray class field $K(p^\infty)$, all the conditions are automatic except possibly \ref{assumption:eltsorderp}, and this may be dealt with by replacing $K_\infty$ by a finite subextension.
    We shall study extensions of this type in more detail in \S \ref{sect:imquad} below, where we take $K$ to be an imaginary quadratic field.
   \end{remark}

   As $K_\infty$ is a Galois extension of $\QQ$, the Galois groups $\Gal(K_{\infty, \fp_i}\slash K_{\fp_i})$, $1\leq i\leq e$, are conjugate to each other in $\Gal(K_\infty\slash \QQ)$, as are their inertia subgroups. If $L_{\infty,i}$ denotes the maximal unramified extension of $K_{\fp,i}$ in $K_{\infty,\fp_i}$, we get canonical identifications of $L_{\infty,i}$ with $L_{\infty,j}$ for all $1\leq i,j\leq e$. We can therefore drop the index and denote this unramified extension of $\Qp$ by $F_\infty$.

   As explained in Section \ref{sect:semilocalregulator}, for $1\leq i\leq e$, we have a regulator map
   \[\cL^G_{V,\fp_i} : Z^1_{\Iw, \fp_i}(K_\infty, T) \rightarrow \DD_{\cris,\fp_i}(V) \otimes_{\Qp} \cH_{\Finf}(G).\]
   Via the localisation map $\loc_{\fp_i}:H^1_{\Iw,S}(K_\infty, T) \rightarrow Z^1_{\Iw, \fp_i}(K_\infty, T)$, it induces a map
   \[  H^1_{\Iw,S}(K_\infty, T) \rTo \cH_{\Finf}(G) \otimes_{\Qp} \DD_{\cris,\fp_i}(V)\]
   which we also denote by $\cL^G_{V,\fp_i}$. Let
   \[
    \DD_p(V) =\Dcris\Big(\big(\Ind_{K\slash\QQ}V\big)|_{\cG_{\Qp}}\Big) \cong \bigoplus_{i=1}^e \DD_{\cris,\fp_i}(V).
   \]
   Define
   \[ \cL^G_V = \bigoplus_{i=1}^e \cL^G_{V,\fp_i} : H^1_{\Iw,S}(K_\infty,V) \rTo \cH_{\Finf}(G) \otimes_{\Qp} \DD_p(V).\]

   Denote by $\mathcal{K}_{\Finf}(G)$ the fraction field of $\mathcal{H}_{\Finf}(G)$. Assume that Conjecture $\Leop(K_\infty,V)$ (as formulated in \S \ref{sect:globalranks} below) olds, so $H^2_{\Iw}(K_\infty\slash K,V)$ is $\Lambda_{\Qp}(G)$-torsion. Let $d = \frac{1}{2}[K : \QQ] \dim_{\Qp}(V)$. As $\rank_{\Lambda_{\Zp}(G)} H^1_{\Iw,S}(K_\infty,T)=d$ by Theorem \ref{thm:globalrank},  the regulator $\cL^G_V$ induces a map\footnote{For the definition of the determinant of a finitely generated $\Lambda_{\Zp}(G)$-module, see \cite{knudsenmumford76}; c.f. also \cite[\S 3.1.5]{perrinriou94}.}
   \[ \det \cL^G_V : \det_{\Lambda_{\Qp}(G)} H^1_{\Iw,S}(K_\infty,V) \rTo \mathcal{K}_{\Finf}(G)\otimes_{\Qp}\bigwedge^d \DD_p(V).\]

   \begin{definition}
    \label{def:Iarith}
    Define $\II_{\arith,p}(V)$ to be the $\Lambda_{\Qp}(G)$-submodule of $\mathcal{K}_{\Finf}(G)\otimes_{\Qp}\bigwedge^d\DD_p(V)$
    \[ \II_{\arith,p}(V)= \det \cL^G_V\big(H^1_{\Iw}(K_\infty,T)\big) \otimes \left(\det H^2_{\Iw}(K_\infty,T)\right)^{-1}.\]
   \end{definition}

   In the spirit of Perrin-Riou (c.f. \cite[\S 3.1]{perrinriou03}), we can give an explicit description of $\II_{\arith,p}(V)$ as follows. Let $f_2 \in \Lambda_{\Qp}(G)$ be a generator of the characteristic ideal of $H^2_{\Iw}(K_\infty,T)$, so $\det H^2_{\Iw}(K_\infty,T) = f_2^{-1} \Lambda_{\Qp}(G)$.

   \begin{proposition}
    \label{prop:explicitdescription}
    Let $\fc=\{c_1,\dots,c_d\}\subset H^1_{\Iw,S}(K_\infty,V)$ be elements such that if $\mathcal{C}$ denotes the $\Lambda_{\Qp}(G)$-submodule of $H^1_{\Iw,S}(K_\infty,V)$ spanned by the elements of $\fc$, then the quotient $H^1_{\Iw,S}(K_\infty,V)\slash \mathcal{C}$ is $\Lambda_{\Qp}(G)$-torsion. Denote by $f_{\fc}\in\Lambda_{\Qp}(G)$ the corresponding characteristic element. Then
   \[  \II_{\arith,p}(V)=\Lambda_{\Qp}(G) \, f_2 f_{\fc}^{-1} \, \cL^G_V(c_1)\wedge\dots\wedge \cL^G_V(c_d).\]
   \end{proposition}

   \begin{proof}
    Clear from the construction.
   \end{proof}

   \begin{remark}
    If $H^1_{\Iw,S}(K_\infty,V)$ is free as a $\Lambda_{\Qp}(G)$-module, then $\II_{\arith,p}(V)$ must be contained in $\cH_{\Finf}(G)$.
   \end{remark}

   \begin{remark}
    Via the isomorphism
    \[ \mathcal{K}_{\Finf}(G)\otimes\bigwedge^d\DD_p(V)\cong \Hom_{\Qp}\Big(\bigwedge^d\DD_p(V^*(1)),\mathcal{K}_{\Finf}(G)\Big),\]
    we can consider the $\Lambda_{\Qp}(G)$-module $\II_{\arith,p}(V)$ as a submodule of
    \[ \Hom_{\Qp}\big(\bigwedge^d\DD_p(V^*(1)),\mathcal{K}_{\Finf}(G)\big).\]
   \end{remark}

   The following proposition implies that $\II_{\arith,p}(V)\neq 0$:

   \begin{proposition}
    \label{prop:smallkernel}
    Assume that conjecture $\Leop(K_\infty,V^*(1))$ holds. Then the kernel of the homomorphism
    \[ \loc_p:H^1_{\Iw,S}(K_\infty, T) \rTo \bigoplus_{i=1}^e Z^1_{\Iw,\fp_i}(K_\infty,T)\]
    is $\Lambda_{\Zp}(G)$-torsion.
   \end{proposition}

   \begin{proof}
    We adapt the arguments in \cite[\S A.2]{perrinriou95}. For $0\leq j\leq 2$, define the $\Lambda_{\Zp}(G)$-modules
    \[ Z^j_S(K_\infty,T)=\bigoplus_{v\in S} H^j_{\Iw}(K_{\infty,v},T)\hspace{3ex} \text{and}\hspace{3ex} Z^j_p(K_\infty,T)= \bigoplus_{i=1}^e Z^j_{\Iw,\fp_i}(K_\infty,T).\]
    Also, define
    \[ X^i_{\infty,S}(K_\infty,T)=H^i\big(G_S(K_\infty),V^*(1)\slash T^*(1)\big). \]
    Taking the limit over the $K_{m,n}$ of the Poitou-Tate exact sequence gives an exact sequence of $\Lambda_{\Zp}(G)$-modules
    \begin{align*}
     0 &\rTo X^2_{\infty,S}(K_\infty,T)^\vee  \rTo H^1_{\Iw,S}(K_\infty,T) \rTo^{\loc_S} Z^1_S(K_\infty,T) \\
        &\rTo  X^1_{\infty,S}(K_\infty,T)^\vee \rTo H^2_{\Iw,S}(K_\infty,T) \rTo Z^2_{S}(K_\infty,T) \\
        &\rTo  X^0_{\infty,S}(K_\infty,T)^\vee\rTo 0.
    \end{align*}

    By Theorem \ref{thm:globalrank}, since we are assuming $\Leop(K_\infty,V^*(1))$, the module $X^2_{\infty,S}(K_\infty,T)$ is $\Lambda_{\Zp}(G)$-cotorsion. Thus $\ker(\loc_S)$ is $\Lambda_{\Zp}(G)$-torsion. As
    \[ \rank_{\Lambda_{\Zp}(G)}Z^1_S(K_\infty,T)=\rank_{\Lambda_{\Zp}(G)}Z^1_p(K_\infty,T)\]
    by Proposition \ref{thm:localrank}, this implies the result.
   \end{proof}

   \begin{corollary}
    If conjecture $\Leop(K_\infty,V)$ holds, the $\Lambda_{\Zp}(G)$-mod\-ule $\II_{\arith,p}(V)$ is non-zero.
   \end{corollary}

   \begin{proof}
    Consequence of Propositions \ref{prop:smallkernel} and \ref{prop:injectivereg}.
   \end{proof}


   As in the cyclotomic case, we conjecture that $\II_{\arith,p}(V)$ should have a canonical basis vector -- a $p$-adic $L$-function for $V$ -- whose image under evaluation at de Rham characters of $G$ is related to the critical $L$-values of $V$ and its twists. In the above generality this is a somewhat vain exercise as even the analytic continuation and algebraicity of the values of the complex $L$-function is conjectural. In the next section, we shall make this philosophy precise in some special cases; we shall show that it is consistent with known results regarding $p$-adic $L$-functions, but that it also implies some new conjectures regarding $p$-adic $L$-functions of modular forms.


 \section{Imaginary quadratic fields}
  \label{sect:imquad}

  \subsection{Setup}
   \label{sect:setupimquad}

   Throughout this section, let $K$ be an imaginary quadratic field in which $p$ splits; write $(p)=\fp \fpb$. We now introduce a specific class of extensions $K_\infty / K$ for which the hypotheses in Section \ref{sect:defmoduleofLfunctions} are satisfied. Let $\ff$ be an integral ideal of $K$ prime to $\fp$ and $\fpb$, and let $K_\infty$ be the ray class field $K(\ff p^\infty)$. We assume that $\ff$ is stable under $\Gal(K / \QQ)$, which is equivalent to the assumption that $K_\infty$ is Galois over $\QQ$. It is well known that $K_\infty \supseteq K(\mu_{p^\infty})$, and that the primes $\fp$ and $\fpb$ are finitely decomposed in $K$; so $G = \Gal(K_\infty / K)$ is an abelian $p$-adic Lie group of dimension 2, and the decomposition groups $G_{\fp}$ and $G_{\fpb}$ are open subgroups.

   \begin{lemma}
    If $p$ is coprime to the order of the ray class group $\operatorname{Cl}_{\ff}(K)$, then $G$ has no elements of order $p$.
   \end{lemma}

   \begin{proof}
    By class field theory, we have an exact sequence
    \[ 0 \rTo \overline{U_\ff} \rTo (\cO_K \otimes \Zp)^\times \rTo \operatorname{Cl}_{\ff p^\infty}(K) \rTo \operatorname{Cl}_{\ff}(K) \rTo 0,\]
    where $U_\ff$ is the group of units of $\cO_K$ that are 1 modulo $\ff$ and $\overline{U_\ff}$ is the closure of $U_\ff$ in $(\cO_K \otimes \Zp)^\times$. So it suffices to show that the quotient
    \[ \frac{(\cO_K \otimes \Zp)^\times}{\overline{U_\ff}}\]
    is $p$-torsion-free. However, since $K$ is imaginary quadratic, $\overline{U_\ff} = U_\ff$ is a finite group, and as $p$ is odd and split in $K$, we have $p \nmid |\overline{U_\ff}|$. Since $(\cO_K \otimes \Zp)^\times$ is $p$-torsion-free, the result follows.
   \end{proof}

   Let us write $\Gamma_\fp = \Gal(K_\infty / K(\ff \fpb^\infty))$ and $\Gamma_{\fpb} = \Gal(K_\infty / K(\ff \fp^\infty))$. Note that $\Gamma_\fp$ and $\Gamma_{\fpb}$ are $p$-adic Lie groups of rank 1 whose intersection is trivial, and the open subgroup $\Gal(K_\infty / K(\ff))$ is isomorphic to $\Gamma_\fp \times \Gamma_{\fpb}$.

   Let $T$ be a $\Zp$-representation of $\cG_K$ of rank $d$, and let $V=T[p^{-1}]$. As in the previous section, assume that for $\star\in \{\fp, \fpb \}$, the restriction of $V$ to $\cG_{K_\star}$ is good crystalline.

   \begin{lemma}
    We have a canonical decomposition
    \[ \bigwedge^d \DD_p(V) \cong \bigoplus_{m+n=d} \left( \bigwedge^m \DD_{\cris,\fp}(V)\otimes_{\Qp} \bigwedge^n \DD_{\cris,\fpb}(V) \right). \]
   \end{lemma}

   \begin{proof}
    Clear from the definition.
   \end{proof}

   To simplify the notation, let us write \[\DD_p(V)^{(m,n)}=\bigwedge^m \DD_{\cris,\fp}(V)\otimes_{\Qp} \bigwedge^n \DD_{\cris,\fpb}(V).\]


  \subsection{Galois descent of the module of L-functions}
   \label{sect:Galoisdescent}

   \begin{definition}
    For $m,n\in\NN$ such that $m+n=d$, define $\II^{(m,n)}_{\arith,p}(V)$ to be the image of $\II_{\arith,p}(V)$ in $\mathcal{K}_{\Finf}(G)\otimes_{\Qp} \DD_p(V)^{(m,n)}$ induced from the projection map $\bigwedge^d \DD_p(V)\rTo \DD_p(V)^{(m,n)}$.
   \end{definition}

   The following theorem is the main result of this section:

   \begin{theorem}
    \label{thm:coefficients}
    If $d = 2n$, then $\II^{(n,n)}_{\arith,p}(V) \subset \mathcal{K}_{L}(G) \otimes_{\Qp} \DD_p(V)^{(n,n)}$, where $L$ is a finite unramified extension of $\Qp$.
   \end{theorem}

   The rest of this section is devoted to proving this theorem. As in Proposition \ref{prop:galoisequivariance}, let $\tilde \sigma_{\fp} \in G_{\fp}$ be the unique element of $G_{\fp}$ which lifts the Frobenius automorphism at $\fp$ of $K(\ff \fpb^\infty)$ and which is trivial on $K(\mu_{p^\infty})$, and similarly for $\fpb$.

   \begin{lemma}
    The element $\tilde \sigma_{\fp} \tilde \sigma_{\fpb} \in G$ has finite order.
   \end{lemma}

   \begin{proof}
    Let us consider the ``semilocal Artin map''
    \[ \theta = (\theta_\fp, \theta_{\fpb}) : K_{\fp}^\times \times K_{\fpb}^\times \to G.\]
    Here $\theta_{\fp}$ is the Artin map for $K_\fp$, normalised so that uniformisers map to geometric Frobenius elements. The kernel of $\theta$ is the image in $K_{\fp}^\times \times K_{\fpb}^\times$ of the elements of $K^\times$ which are units outside $p$ and congruent to $1 \bmod \ff$.

    By the functoriality of the global Artin map (cf.~\cite[VI.5.2]{neukirch99}), there is a commutative diagram
    \[
     \begin{diagram}
      K_{\fp}^\times \times K_{\fpb}^\times & \rTo^\theta & G \\
      \dTo^{N_{K/\QQ}} & & \dTo \\
      \Qp^\times & \rTo & \Gamma,
     \end{diagram}
    \]
    The bottom horizontal map is the local Artin map for $\QQ(\mu_{p^\infty}) / \QQ$; if we identify $\Gamma$ with $\Zp^\times$, this map is the identity on $\Zp^\times$ and sends $p$ to 1.

    Consider the element $(p, 1)$ of $K_{\fp}^\times \times K_{\fpb}^\times$. The image of this in the group $\Gal(K(\ff \fpb^\infty) / K)$ is the Frobenius $\sigma_p$. Its image in $\Qp^\times$ is $p$, which is mapped to the identity in $\Gamma$. Hence the image of $(p, 1)$ in $G$ is $\tilde \sigma_\fp$. Similarly, $(1, p)$ is a lifting of $\tilde \sigma_{\fpb}$.

    Hence $\tilde \sigma_{\fp} \tilde \sigma_{\fpb}$ is the image of the element $(p, p)$ of $K_{\fp}^\times \times K_{\fpb}^\times$. Thus if $m$ is such that $p^m = 1 \bmod \ff$, $(p^m, p^m) \in K_{\fp}^\times \times K_{\fpb}^\times$ is in the kernel of $\theta$, and hence $(\tilde \sigma_{\fp} \tilde \sigma_{\fpb})^m = 1$ in $G$.
   \end{proof}

   \begin{corollary}
    \label{cor:coefficients}
    Let $x_1,\dots,x_n$ be any elements of $Z^1_{\fp}(K_\infty, V)$, and similarly let $y_1,\dots,y_n \in Z^1_{\fpb}(K_{\infty}, V)$. Then the element
    \[ \left(\cL^G_{V, \fp}(x_1) \wedge \dots \wedge \cL^G_{V, \fp}(x_n) \right) \otimes \left(\cL^G_{V, \fpb}(y_1) \wedge \dots \wedge \cL^G_{V, \fpb}(y_n) \right) \]
    of $ \cH_{\Finf}(G) \otimes_{\Qp} \DD_p(V)^{(n, n)}$
    in fact lies in
    \( \cH_{L}(G) \otimes_{\Qp} \DD_p(V)^{(n, n)},\)
    where $L$ is a finite unramified extension of $\Qp$.
   \end{corollary}

   \begin{proof}
    Clear, since the Frobenius automorphism of $\Finf$ acts on this element as multiplication by $[\tilde{\sigma}_{\fp} \tilde{\sigma}_{\fpb}]^n$, which we have seen has finite order.
   \end{proof}

   \begin{remark}
    As is clear from the proof of the lemma and its corollary, the degree of $L / \Qp$ is bounded by the exponent of the ray class group of $K$ modulo $\ff$, and in particular is independent of $V$.
   \end{remark}

   We deduce Theorem \ref{thm:coefficients} by combining Corollary \ref{cor:coefficients} with Proposition \ref{prop:explicitdescription}.

  \subsection{Orders of distributions}

   Let us choose subspaces $W_\fp \subseteq \bigwedge^m \Dcris(K_\fp, V)$ and $W_{\fpb} \subseteq \bigwedge^n \Dcris(K_\fpb, V)$. Then the space
   \[ Q = \left( \bigwedge^m \Dcris(K_\fp, V) / W_\fp \right) \otimes_{\Qp} \left( \bigwedge^n \Dcris(K_\fpb, V) / W_\fpb \right)\]
   is a quotient of $\DD_p(V)^{(m, n)}$ and hence of $\DD_p(V)$. So, for any $c_1, \dots, c_d \in H^1_{\Iw, S}(K_\infty, V)$, we may consider the projection of
   \[ \cL^G_V(c_1, \dots, c_d) = \cL^G_V(c_1) \wedge \dots \wedge \cL^G_V(c_d)\]
   to $Q$.

   \begin{theorem}
    \label{thm:orderofproduct}
    The distribution $\pr_Q\left(\cL^G_V(c_1) \wedge \dots \wedge \cL^G_V(c_d)\right)$ is a distribution on $G$ of order $(m h_\fp, n h_{\fpb})$ with respect to the subgroups $(\Gamma_{\fp}, \Gamma_{\fpb})$, where $h_\fp$ (resp. $h_{\fpb}$)  is the largest valuation of any eigenvalue of $\varphi$ on $\wedge^m \Dcris(K_\fp, V) / W_\fp$ (resp. on $\wedge^n \Dcris(K_\fpb, V) / W_{\fpb}$).
   \end{theorem}

   \begin{proof}
    Let us write $c_{j, \fp}$ for the localisation of $c_j$ at $\fp$, and similarly for $\fpb$. By Proposition \ref{prop:orderofdistributions}, for each subset $\{j_1, \dots, j_m\} \subseteq \{1, \dots, d\}$ of order $m$, the projection of the element
    \[ \cL_{V, \fp}^G(c_{j_1,\fp}) \wedge \dots \wedge \cL_{V, \fp}^G(c_{j_m, \fp})\]
    to $\bigwedge^m \Dcris(K_\fp, V) / W_{\fp}$ is a distribution of order $(h_{\fp}, 0)$ with respect to the subgroups $(\Gamma_\fp, U)$, where $U = \Gal(K_\infty / K(\mu_{p^\infty}))$. By the change-of-variable result of Proposition \ref{prop:changevar} in the appendix, it is also a distribution of order $(h_{\fp},0)$ with respect to the subgroups $(\Gamma_{\fp}, \Gamma_{\fpb})$.

    We have also a corresponding result for the projection to $\bigwedge^n \Dcris(K_\fpb, V) / W_{\fpb}$ of the distribution obtained from any $n$-element subset of $\{1, \dots, d\}$: this gives a distribution with order $(0, h_{\fpb})$ with respect to $(\Gamma_\fp, \Gamma_\fpb)$. Since the product of distributions of order $(a,0)$ and $(0, b)$ is a distribution of order $(a, b)$ by Proposition \ref{prop:orderofproduct}, the product of any two such subsets gives a distribution with values in $Q$ of order $(h_\fp, h_{\fpb})$. Since $\pr_Q\left(\cL^G_V(c_1) \wedge \dots \wedge \cL^G_V(c_d)\right)$ is a finite linear combination of products of this form, the theorem follows.
   \end{proof}


  \subsection{Example 1: Gr\"ossencharacters and Katz's L-function}
   \label{sect:KatzLfunction}


   \subsubsection{Kummer maps}

    We recall the well-known local theory of exponential maps for the representation $\Zp(1)$. For any finite extension $L / \Qp$, there is a Kummer map $\kappa_L: \cO_L^\times \to H^1(L, \Zp(1))$, whose kernel is the Teichm\"uller lifting of $k_L^\times$. In particular, the restriction to the kernel $U^1(L)$ of reduction modulo the maximal ideal is an injection.

    Moreover, after inverting $p$, we have a commutative diagram relating the Kummer map to the exponential map of Bloch--Kato (see \cite{blochkato90}):
    \[
     \begin{diagram}
      \Qp \otimes_{\Zp} U^1(L) & \rTo^{\kappa_L} & H^1(L, \Qp(1)) \\
      \dTo & & \dEq\\
      \DD_{\dR, L}(\Qp(1)) &\rTo^{\exp_{L, \Qp(1)}} & H^1(L, \Qp(1))
     \end{diagram}
    \]
    where the vertical map sends
    $u$ to $t^{-1} \log(u) \otimes e_1$, where $e_1$ is the basis vector of $\Qp(1)$ corresponding to our compatible system of roots of unity.

    The maps $\kappa_L$ are compatible with the norm and corestriction maps for finite extensions $L' / L$, so for an infinite algebraic extension $K_\infty / \Qp$ we can take the inverse limit over the finite extensions of $\Qp$ contained in $K_\infty$ to define
    \[ \kappa_{K_\infty} : U^1(K_\infty) \rTo H^1_{\Iw}(K_\infty, \Qp(1)),\]
    where $U^1(K_\infty) := \varprojlim_{K' \subset K_\infty} U^1(K')$.


   \subsubsection{Coleman series}

    We recall the following basic result, due to Coleman. Let $\cF$ be any height 1 Lubin--Tate group over $\Qp$, and $F$ an unramified extension of $\Qp$. Fix a generator $v = (v_n)$ of the Tate module of $\cF$ (that is, a norm-compatible sequence of $p^n$-torsion points of $\cF$).

    \begin{theorem}[{\cite{coleman79}}]
     \label{thm:colemanstheorem}
     Let $F$ be a finite unramified extension of $\Qp$. Then for each $\beta = (\beta_n) \in U^1(F(\cF_{p^\infty}))$, there is a unique power series
     \[ g_{F,\cF}(\beta) \in \cO_F[[X]]^{\times,N_\cF = 1}\]
     where $N_\cF$ is Coleman's norm operator, such that for all $n \ge 1$ we have
     \[ \beta_n = [g_{F,\cF}(\beta)]^{\sigma^{-n}}(v_n).\]
    \end{theorem}

    Here $\sigma$ is the arithmetic Frobenius automorphism of $F / \Qp$, which we extend to an automorphism of $\cO_F[[X]]$ acting trivially on the variable $X$.

    If $\cF$ is the formal multiplicative group $\hat{\GG}_m$, then we shall drop the suffix $\cF$; and we take $v_n = \zeta_n - 1$, where $(\zeta_n)$ is our chosen compatible sequence of $p$-power roots of unity. In this case, if we identify $X$ with the variable $\pi$ in Fontaine's rings, the relation between the map $g_F$ and the Perrin-Riou regulator map is given by the following diagram:
    \[
     \begin{diagram}
      U^1(F(\mu_{p^\infty})) & \rTo^{\kappa_{F(\mu_{p^\infty})}} & H^1_{\Iw}(F(\mu_{p^\infty}), \Zp(1))\\
      \dTo^{g_F} & & \\
      \cO_F[[ \pi ]]^{\times,N = 1} & & \dTo^{\cL_{F, \Qp(1)}^\Gamma}\\
      \dTo^{(1 - \tfrac{\vp}{p})\log} & & \\
      \cO_F[[ \pi ]]^{\psi = 0} & \rTo & \cH(\Gamma) \otimes_{\Qp} \Dcris(F, \Qp(1))
     \end{diagram}
    \]

    If we identify $\Dcris(F, \Qp(1))$ with $F$ via the basis vector $t^{-1} \otimes e_1$, then the bottom map sends $f \in \cO_F[[ \pi ]]^{\psi = 0}$ to $\ell_0 \cdot \mathfrak{M}^{-1}(f)$, where $\ell_0 = \frac{\log \gamma}{\log \chi(\gamma)}$ for any non-identity element $\gamma \in \Gamma_1$ and $\mathfrak{M}$ is the Mellin transform as defined in Section \ref{sect:Fontainerings}. (See e.g.~the proof of Proposition 1.5 of \cite{leiloefflerzerbes11}.) Thus the image of the bottom map is precisely $\ell_0 \cdot \Lambda_{\cO_F}(\Gamma) \subseteq \cH_F(\Gamma)$; and if we define
    \[ h_F(\beta) = \ell_0^{-1}\cdot  \cL_{F, \Qp(1)}^\Gamma(\kappa_{F_\infty}(\beta)) \in \Lambda_{\cO_F}(\Gamma),\]
    then we have
    \[ \mathfrak{M}(h_F(\beta)) = (1 - \tfrac{\vp}{p})\log g_F(\beta).\]


   \subsubsection{Two-variable Coleman series}

    Now let $K_\infty / \Qp$ be an abelian $p$-adic Lie extension containing $\Qp(\mu_{p^\infty})$ such that $G=\Gal(K_\infty\slash\Qp)$ is a $p$-adic Lie group of dimension $2$. Let $\Finf$ be the completion of the maximal unramifed subextension of $K_\infty$. We define
    \[ h_{\infty} :  U^1(K_\infty) \to \Lambda_{\OFinf}(G)\]
     to be the unique map such that the composite
    \[ U^1(K_\infty) \rTo^{\kappa_{K_\infty}} H^1_{\Iw}(K_\infty/\Qp, \Zp(1)) \rTo^{\cL_{\Qp(1)}^G} \cH_{\Finf}(G)\]
    is equal to $\ell_0 \cdot h_\infty$.

    \begin{proposition}
     \label{prop:colemanmaps}
     The element $h_\infty(\beta)$ is uniquely determined by the relation
     \[ h_\infty(\beta) = \sum_{\sigma \in U_F} h_{F}(\beta_F)^\sigma [\sigma] \pmod{I_F}\]
     for all unramified subextensions $F \subset K_\infty$, where $U_F = \Gal(F / \Qp)$, $I_F$ is the kernel of the natural map $\Lambda_{\OFinf}(G) \to \Lambda_{\OFinf}(U_F \times \Gamma)$, and $\beta_F$ denotes the image of $\beta$ in $U^1(F(\mu_{p^\infty}))$.
    \end{proposition}

    \begin{proof}
     This follows from the compatibility of the maps $\cL^\Gamma$ and $\cL^G$ (Theorem \ref{thm:localregulator}(i)).
    \end{proof}

    We would like to compare this result to Theorem 5 of \cite{yager82}. Our method differs from that of Yager, as we build measures on $G$ out of measures on the Galois groups of extensions $F(\mu_{p^\infty}) / F$ for unramified extensions $F \subset K_\infty$, while Yager considers instead the extensions $F(\mathcal{F}_{p^\infty})/F$ where $\cF$ is the Lubin--Tate group corresponding to an elliptic curve with CM by $\cO_K$.

    Let $\cF$ be any Lubin--Tate formal group over $\Qp$ which becomes isomorphic to $\hat{\GG}_m$ over $\Finf$. If $F$ is any finite unramified extension of $\Qp$ contained in $K_\infty$, then $F(\cF_{p^\infty}) \subseteq K_\infty$. For any $\beta \in U^1(K_\infty)$, let $\beta_{F, \cF}$ be its image in $U^1(F(\cF_{p^\infty}))$. Then Coleman's theorem (Theorem \ref{thm:colemanstheorem}) gives us an element
    \[ g_{F, \cF}(\beta_{F, \cF}) \in \cO_F[[X]]^{\times,N_\cF = 1}.\]

    We write
    \[ h_{F, \cF}(\beta_{F, \cF}) = \mathfrak{M}^{-1} \left[ \left(1 - \tfrac{\vp}{p}\right) \log\left( g_{F, \cF}(\beta) \circ \theta\right)\right] \]
    where $\theta$ is the unique power series in $\OFinf[[X]]$ giving an isomorphism $\cF \rTo^\cong \hat{\GG}_m$ such that $v_n$ maps to $\zeta_n - 1$.

    \begin{theorem}[de Shalit]
     We have
     \[ h_\infty(\beta) = \sum_{\sigma \in U_F} h_{F, \cF}(\beta_{F,\cF})^\sigma [\sigma] \pmod{I_{F, \cF}},\]
     where $I_{F, \cF}$ is the kernel of the natural map
     \[ \Lambda_{\OFinf}(G) \rTo \Lambda_{\OFinf}( \Gal(F(\cF_{p^\infty}) / F)).\]
    \end{theorem}

    \begin{proof}
     See \cite[\S I.3.8]{deshalit87}. (Note that the theorem is stated there for $K_\infty = \QQ_p^{\mathrm{ab}}$, the maximal abelian extension of $\Qp$; but the theorem, and the proof given, are true with $K_\infty$ replaced by any smaller extension over which the formal groups concerned become isomorphic.)
   \end{proof}

    It follows that the map $h_{\infty}$ defined above coincides with the map constructed (under more restrictive hypotheses) by Yager in \cite{yager82}. In particular, if $c$ denotes the element of the global $H^1_{\Iw}$ obtained by applying the Kummer map to the elliptic units, then $\cL^G_{\fp, V} (c)$ is equal to $\ell_0 \mu$ where $\mu$ is Katz's $p$-adic $L$-function.


  \subsection{Example 2: Two-variable L-functions of modular forms}

   We now consider the restriction to $\cG_{K}$ of the representation $V$ of $\cG_\QQ$ attached to a modular form $f$ of weight 2, level $N$ prime to $p \Delta_{K/\QQ}$, and character $\delta$. Let $E \subseteq \Qpb$ be the completion of $\QQ(f) \subseteq \Qb$ at our chosen prime of $\Qb$. We take $V = V_f^*$, so $V$ has Hodge--Tate weights $\{0, 1\}$ at each of $\fp$ and $\fpb$. Let $\{\alpha, \beta\}$ be the roots of $X^2 - a_p X + p \delta(p)$, so the eigenvalues of $\vp$ on either $\Dcris(K_\fp, V)$ or $\Dcris(K_\fpb, V)$ are $\alpha^{-1}$ and $\beta^{-1}$.

   \begin{definition}
    A $p$-refinement of $f$ is a pair $u = (u_\fp, u_{\fpb}) \subseteq \{\alpha, \beta\}^{\times 2}$. We say that $u$ is \emph{non-critical} if $v_p(u_\fp), v_p(u_{\fpb}) < 1$; otherwise $u$ is \emph{critical}.
   \end{definition}

   Let $K_\infty$ be an extension of $K$ with Galois group $G$, satisfying the hypotheses specified in Section \ref{sect:setupimquad}. For a finite-order character $\omega$ of $G$, let $L_{\{\fp, \fpb\}}(f/K, \omega^{-1}, s)$ denote the twisted $L$-function of $f$ with the Euler factors at $\fp$ and $\fpb$ removed. Let $\Omega^+_f$ and $\Omega^-_f$ be the real and complex periods of $f$ (which are defined up to multiplication by an element of $\QQ(f)^\times$).

   \begin{conjecture}[Existence of $L$-functions]
    \label{conj:lfunc}
    Let $(u_\fp, u_\fpb)$ be a $p$-refinement which is non-critical. Then there exists a distribution $\mu_f(u_\fp, u_\fpb)$ on $G$, of order $(v_p(u_\fp), v_p(u_{\fpb}))$ with respect to the subgroups $(\Gamma_\fp, \Gamma_{\fpb})$, such that for all finite-order characters $\omega$ we have
    \begin{align}
     & \int_{G} \omega\, \mathrm{d}\mu_f(u_\fp, u_\fpb) = \notag \\
     & \left( \prod_{q \in \{\fp, \fpb\}} u_q^{-c_q(\omega)} e_q(\omega^{-1}) \frac{P_q(\omega^{-1}, u_q^{-1})}{P_q(\omega, p^{-1} u_q)} \right)
     \frac{L_{\{\fp, \fpb\}} (f / K, \omega^{-1}, 1)}{\Omega^+_f \Omega^-_f}.
     \label{eq:interpolating}
    \end{align}
   \end{conjecture}

   \begin{remark}
    The definition of the order of a distribution on $\Zp^2$ is given in Section \ref{sect:order}. The hypothesis that the $p$-refinement be non-critical implies that the distribution $\mu_f(u_\fp, u_\fpb)$ is unique if it exists, since a distribution of order $(r,s)$ with $r,s < 1$ is uniquely determined by its values at finite-order characters.
   \end{remark}

   Two approaches are known to the construction of such $L$-functions: either via $p$-adic interpolation of Rankin--Selberg convolutions, as in \cite{hida88,perrinriou88,kim-preprint}, or via the combinatorics of modular symbols on symmetric spaces attached to $\GL(2, \AA_K)$, as in \cite{haran87}. The details have not been written down in the full generality described above (although M.~Emerton and B.~Zhang have announced results of this kind in a paper which is currently in preparation). The literature to date contains constructions of $\mu_f(u_\fp, u_\fpb)$ in the following cases:
   \begin{itemize}
    \item if $f$ is ordinary, $\delta = 1$, and $u$ is the ``ordinary refinement'' $(\alpha, \alpha)$ where $\alpha$ is the unit root \cite{haran87}
    \item if $f$ is ordinary, $u$ is the ordinary refinement, and $G$ decomposes as a direct product of eigenspaces for complex conjugation \cite{perrinriou88}
    \item if $f$ is non-ordinary, $u_{\fp} = u_{\fpb}$, $[K(\ff p) : K]$ is prime to $p$, $\delta^2 = 1$, and we consider only the restriction of the distribution to the set of characters whose restriction to $\Gal(K(\ff p) / K)$ does not factor through a Dirichlet character via the norm map \cite{kim-preprint}.
   \end{itemize}

   \begin{remark}\mbox{~}
    (i) We have chosen to write the interpolating formula \eqref{eq:interpolating} in a way that emphasises the similarity with that of \cite{cfksv}. The cited references use a range of different formulations, and the distributions they construct differ from ours by various correction factors; but in each case the \emph{existence} of a measure satisfying their conditions is equivalent to the conjecture above.

     (ii) If $f$ is ordinary and $u$ is the ordinary refinement, the condition that $\mu_f(u)$ has order $(0, 0)$ is simply that it be a measure. In the non-ordinary case considered by Kim, the condition that $\mu_f(u)$ has order $(v_p(u_\fp), v_p(u_{\fpb}))$ is more delicate, and depends crucially on the decomposition of $\Gal(K_\infty / K(\ff))$ as the direct product of the distinguished subgroups $\Gamma_{\fp}$ and $\Gamma_{\fpb}$ corresponding to the two primes above $p$.
   \end{remark}

   We now give a conjectural interpretation of these $p$-adic $L$-functions in terms of our regulator map $\cL^G_V$. Let us write
   \[ Z^1_{\Iw, p}(V) = Z^1_{\Iw, \fp}(V) \oplus Z^1_{\Iw, \fpb}(V).\]
   We write $\exp^*_{V}$ for the map $\exp^*_{K_\fp, V} \oplus \exp^*_{K_\fpb, V} : Z^1_{\Iw,p}(V) \to \DD_p(V)$, and similarly $\cL^G_V$ for the map
   $\cL^G_{\fp, V} \oplus \cL^G_{\fpb, V} : Z^1_{\Iw, p}(V) \to \cH_{\Finf}(G) \otimes \DD_p(V)$. Both of these induce maps on the wedge square, which we denote by the same symbols.

   The following conjecture can be seen as a special case of the very general ``$\zeta$-isomorphism conjecture'' of Fukaya and Kato (Conjecture 2.3.2 of \cite{fukayakato06}), applied to the module $\Lambda_{\Zp}(G) \otimes T$ for $T$ a $\Zp$-lattice in $V$.

   \begin{conjecture}
    \label{conj:twovarmf}
    Choose a basis $v$ of $\Fil^0 \Dcris(K_\fp, V) \otimes_{\Qp} \Fil^0 \Dcris(K_\fpb, V) \subseteq \DD_p(V)^{(1, 1)}$. Then there is a distinguished element $\fc \in \bigwedge^2 H^1_{\Iw, S}(K_\infty, V_f)$ such that for all finite-order characters $\omega$, we have
    \[ \exp^*_{V(\omega^{-1})^*(1)}(\fc_\omega) = \frac{L(f / K, \omega^{-1}, 1)}{ \Omega^+_f \Omega^-_f }\, v.\]
    Moreover, $\fc$ is a $\Lambda_{\Zp}(G)$-basis of $\II_{\arith,p}(V)$.
   \end{conjecture}

   We choose a basis $v_{\fp, \alpha}, v_{\fp, \beta}$ of $\vp$-eigenvectors in $\Dcris(K_\fp, V)$, and similarly for $\Dcris(K_\fpb, V)$; and for a $p$-refinement $u = (u_\fp, u_{\fpb})$, we let $v_u = v_{\fp, u_{\fp}} \otimes v_{\fpb, u_\fpb} \in \DD_p(V)^{(1, 1)}$. We may normalise such that $v_{\fp} = v_{\fp, \alpha} + v_{\fp, \beta}$ is a basis of $\Fil^0 \Dcris(K_\fp, V)$ (and respectively for $\fpb$); then $v = v_{\fp} \otimes v_{\fpb}$ is a basis of $\Fil^0 \DD_p(V)$.

   \begin{proposition}
    Let $\fc \in \bigwedge^2 H^1_{\Iw, S}(K_\infty, V)$. Then for each $p$-refinement $u$ (critical or otherwise), the projection of $\cL^G_V(\fc)$ to the subspace $E \cdot v_u \subseteq \DD_p(V)^{(1, 1)}$ is a distribution of order $(v_p(u_\fp), v_p(u_{\fpb}))$. If $\fc$ satisfies the condition of Conjecture \ref{conj:twovarmf}, then the projection of $\cL^G_V(\fc)$ satisfies the interpolating property \eqref{eq:interpolating}.
   \end{proposition}

   \begin{proof}
    The values of $\cL^G_{V}(\fc)$ at $\omega$ can be expressed in terms of those of the dual exponential map using Proposition \ref{thm:explicitformula}, which clearly gives the formula of \eqref{eq:interpolating}.

    The statement regarding the orders of the projections is an instance of Theorem \ref{thm:orderofproduct}. Concretely, suppose we choose elements $\fc_1, \fc_2$ such that $\fc_1 \wedge \fc_2 = \fc$. Then we have
    \begin{multline*}
     \cL^G_{V}(\fc) = (v_{\fp, \alpha}\cL_{V, \fp}(\fc_1)_{\alpha} + v_{\fp, \beta}\cL_{V, \fp}(\fc_1)_{\beta} + v_{\fpb, \alpha}\cL_{V, \fpb}(\fc_1)_{\alpha} + v_{\fpb, \beta}\cL_{V, \fpb}(\fc_1)_{\beta}) \\ \wedge (v_{\fp, \alpha}\cL_{V, \fp}(\fc_2)_{\alpha} + v_{\fp, \beta}\cL_{V, \fp}(\fc_2)_{\beta} + v_{\fpb, \alpha}\cL_{V, \fpb}(\fc_2)_{\alpha} + v_{\fpb, \beta}\cL_{V, \fpb}(\fc_2)_{\beta}),
    \end{multline*}
    so the projection of $\cL^G_{V}(\fc)$ to the line spanned by $v_u$ is
    \[ v_u \cdot
     \begin{vmatrix}
      \cL^G_{V, \fp}(\fc_1)_{u_\fp} & \cL^G_{V, \fp}(\fc_2)_{u_\fp}\\
      \cL^G_{V, \fpb}(\fc_1)_{u_\fpb} & \cL^G_{V, \fpb}(\fc_2)_{u_\fpb}
     \end{vmatrix}.
    \]
    Since $\cL^G_{V, \fp}(\fc_?)_{u_\fp}$ (for $? \in \{1, 2\}$) is a distribution of order $(v_p(u_\fp), 0)$, and $\cL^G_{V, \fpb}(\fc_?)_{u_\fpb}$ is a distribution of order $(0, v_p(u_\fpb))$, the determinant gives a distribution of order $(v_p(u_\fp), v_p(u_{\fpb})$ as claimed.
   \end{proof}

   In particular, when the refinement $u$ is non-critical, we conclude that Conjecture \ref{conj:twovarmf} implies Conjecture \ref{conj:lfunc} and the projection of $\cL^G_{V}(\fc)$ to $v_u$ must be equal to the uniquely determined distribution $\mu_f(u)$.

   \begin{remark}
    If Conjecture \ref{conj:twovarmf} holds, then one can also project the element $\cL^G_V(\fc)$ into $\DD_p(V)^{(2, 0)}$ (or into $\DD_p(V)^{(0, 2)}$). The resulting distributions are of a rather simpler type: if $\fc = \fc_1 \wedge \fc_2$ as before, then
    \[
     \pr_{2, 0} \cL^G_{V}(\fc) = \cL^G_{\fp, V}(\fc_1) \wedge \cL^G_{\fp, V}(\fc_2).
    \]
    This is a distribution on $G$ with values in the 1-dimensional space $\DD_p(V)^{(2, 0)} = \det_{\Qp} \Dcris(K_\fp, V)$ of order $(1, 0)$, divisible by the image in $\cH_{\Finf}(G)$ of the distribution $\ell_0 \in \cH_{\Qp}(\Gamma_\fp)$, so dividing by this factor gives a bounded measure on $G$ with values in $\Finf$. Note that acting by the arithmetic Frobenius of $\Finf$ on this measure corresponds to multiplication by $[\sigma_\fp]^2$, so it \emph{never} descends to a finite extension of $\Qp$.

    It is natural to conjecture (and would follow from Conjecture 2.3.2 of \cite{fukayakato06}) that if $\tau$ is a character of $G$ whose Hodge--Tate weights at $\fp$ and $\fpb$ are $(r,s)$ with $r \ge 1$ and $s \le -1$, so $\Fil^0 \bigwedge^2 \DD_{p}(V(\tau^{-1})) = \DD_p(V)^{(2,0)}$, then the value of $\pr_{2, 0} \cL^G_V(\fc)$ at $\tau$ should (after dividing by an appropriate period) correspond to the value at $1$ of the $L$-function of the automorphic representation $\operatorname{BC}(\pi_f) \otimes \tau$ of $\operatorname{GL}(2,\mathbb{A}_K)$. Up to a shift by the cyclotomic character, this corresponds to the set of characters denoted by $\Sigma^{(2)}(\ff)$ in \cite{BDP13}, while the finite-order characters covered by the interpolating property in Conjecture \ref{conj:twovarmf} correspond to the set denoted there by $\Sigma^{(1)}(\ff)$.

    If this conjecture holds, the image of $\pr_{2, 0} \cL^G_{V}(\fc) / \ell_0$ in the Galois group of the anticyclotomic $\Zp$-extension of $K$ should be related to the $L$-functions of \cite[Proposition 6.10]{BDP13} and \cite{brakocevic}, which interpolate the $L$-values of twists of $f$ by anticyclotomic characters in $\Sigma^{(2)}(\ff)$. We intend to study this question further in a future paper.
   \end{remark}

\appendix

 \section{Local and global Iwasawa cohomology}
  \label{appendix:iwacoho}

  In this section, we shall recall some results on the structure of Iwasawa cohomology groups of $p$-adic Galois representations over towers of representations of local and global fields. These are generalizations of well-known results for cyclotomic towers due to Perrin-Riou (cf.~\cite[\S 2]{perrinriou92}); much more general results have since been obtained by Nekovar \cite{nekovar06} and we briefly indicate how to derive the results we need from those of \emph{op.cit.}.

  \subsection{Conventions}

   We shall work with extensions of (local or global) fields $F_\infty/ F$ whose Galois group is of the form $G = \Delta \times \Zp^e$, where $e \ge 1$ and $\Delta$ is a finite abelian group of order prime to $p$. The Iwasawa algebra $\Lambda_{\Zp}(G)$ is a reduced ring, but it is not in general an integral domain; rather, it is isomorphic to the direct product of the subrings $e_\eta \Lambda_{\Zp}(G)$, where $\eta$ ranges over the $\Qpb / \Qp$-conjugacy classes of characters of $\Delta$. For each such $\eta$, $e_\eta \Lambda_{\Zp}(G)$ is a local integral domain.

   In order to greatly simplify the presentation of our results, we shall adopt a minor abuse of notation, following the conventions of \cite{perrinriou95}.

   \begin{definition}
    We shall say that a $\Lambda_{\Zp}(G)$-module has rank $r$ if $M_\eta$ has rank $r$ over $e_\eta \Lambda_{\Zp}(G)$ for all $\eta$.
   \end{definition}

   When using this notation it is important to bear in mind that when $\Delta$ is not trivial, most finitely generated $\Lambda_{\Zp}(G)$ modules will not have a rank.
%

%
%


  \subsection{The local case}
   \label{sect:localranks}

   Let $F$ be a finite extension of $\QQ_\ell$, for some prime $\ell$. Let $V$ a $\Qp$-representation of $\mathcal{G}_{F}$ of dimension $d$, and choose a Galois invariant $\Zp$-lattice $T$. For $F_\infty / F$ an abelian extension satisfying the conditions above, we define
   \[ H^i_{\Iw}(F_\infty, T) = \varprojlim_{K} H^i(K, T)\]
   where the limit is over all finite extensions $K/F$ contained in $F_\infty$, with respect to the corestriction maps; and $H^i_{\Iw}(F_\infty, V) = \Qp \otimes_{\Zp} H^i_{\Iw}(F_\infty, T)$.

   \begin{theorem}
    \label{thm:localrank}
    The groups $H^i_{\Iw}(F_\infty, T)$ are finitely-generated $\Lambda_{\Zp}(G)$-modules, zero if $i \ne \{0, 1\}$. We have an isomorphism
    \[ H^2_{\Iw}(F_\infty, T) \cong H^0(F_\infty, T^\vee(1))^\vee,\]
    where $(-)^\vee$ denotes the Pontryagin dual; in particular $H^2_{\Iw}(F_\infty, T)$ is $\Lambda_{\Zp}(G)$-torsion.

    The group $H^1_{\Iw}(F_\infty, T)$ has well-defined rank given by
    \[ \operatorname{rk}_{\Lambda_{\Zp}(G)} H^1_{\Iw}(F_\infty, T) = \begin{cases} 0& \text{if $\ell \ne p$,} \\ [F : \Qp] d & \text{if $\ell = p$}.\end{cases}\]
   \end{theorem}

   \begin{proof}
    We have assumed that $G$ has a subgroup isomorphic to $\Zp^e$ with $e \ge 1$; thus the profinite degree of $F_\infty / F$ is divisible by $p^\infty$, so $H^0_{\Iw}(F_\infty, T) = 0$ by \cite[8.3.5 Proposition]{nekovar06}.

    For the finiteness statements for $i > 0$, we note that
    \[ H^i_{\Iw}(F_\infty, T) \cong H^i(F, \Lambda_{\Zp}(G) \otimes_{\Zp} T)\]
    by \cite[8.4.4.2 Proposition]{nekovar06}, where the action of $\cG_F$ on $\Lambda_{\Zp}(G)$ is via the inverse of the canonical character $\cG_F \to G \to \Lambda_{\Zp}(G)^\times$. This implies the finite generation of the groups $H^i_{\Iw}(F_\infty, T)$, and their vanishing for $i \ge 3$, by Proposition 4.2.2 of \emph{op.cit.}.

    The isomorphism $H^2_{\Iw}(F_\infty, T)^\vee \cong H^0(F_\infty, T^\vee(1))$ follows by applying local Tate duality to each finite extension $K / F$ contained in $F_\infty$. Finally, the formula for the rank of $H^1_{\Iw}(F_\infty, T)$ follows from Tate's local Euler characteristic formula for finite modules and Corollary 4.6.10 of \emph{op.cit.}.
   \end{proof}


  \subsection{The global case}
   \label{sect:globalranks}

   We now let $K$ be a number field.
   Let $V$ be a $\Qp$-representation of $\mathcal{G}_K$ of dimension $d$, and choose a $G_K$-invariant $\Zp$-lattice $T$. Let $S$ be a finite set of places of $K$ containing all the primes above $p$, all infinite places and all the places whose inertia group acts non-trivially on $V$, and let $K^S$ be the maximal extension of $K$ unramified outside $S$.

   \begin{theorem}[Tate's global Euler characteristic formula]
    If $M$ is a $\Zp$-module of finite length with a continuous action of $\Gal(K^S / K)$, then the modules $H^i(K^S / K, M)$ are finite groups, zero for $i \ge 3$. If $K$ is totally complex, then we have
    \[ \prod_{i =0}^2 \left( \# H^i(K^S / K, M)\right)^{(-1)^i} = (\# M)^{-\tfrac{1}{2}[K : \QQ]}.\]
   \end{theorem}

   \begin{proof}
    See \cite[8.3.17, 8.6.14]{nsw}.
   \end{proof}

   We now consider a Galois extension $K_\infty / K$, contained in $K^S$, whose Galois group $G$ is of the form $\Delta \times \Zp^e$, where $e\geq 1$ and $\Delta$ is abelian of order prime to $p$, as above. For $i \ge 0$, we define
   \[ H^i_{\Iw, S}(K_\infty, T) = \varprojlim_{L} H^i(K^S / L, T)\]
   where the limit is taken over number fields $L$ satisfying $K \subseteq L \subset K_\infty$, with respect to the corestriction maps.

   \begin{theorem}
    \label{thm:globalrank}
    The groups $H^i_{\Iw, S}(K_\infty, T)$ are finitely-generated $\Lambda_{\Zp}(G)$-modules, zero if $i = 0$ or $i \ge 3$. If $K$ is totally complex, then for each character $\eta$ of $\Delta$ we have
    \[ \rank_{e_\eta \Lambda_{\Zp}(G)} e_\eta H^1_{\Iw, S}(K_\infty, T) = \tfrac 1 2 [K : \QQ] d + \rank_{e_\eta \Lambda_{\Zp}(G)}  e_\eta H^2_{\Iw, S}(K_\infty, T).\]
   \end{theorem}

   \begin{proof}
    This follows exactly as in Theorem \ref{thm:localrank}, using Tate's global Euler characteristic formula in place of the local one. (There are no issues with real embeddings, thanks to our running assumption that $p$ be odd.)
   \end{proof}
%

   \begin{proposition}
    \label{prop:leopoldt}
    The following statements are equivalent:
    \begin{enumerate}[(i)]
     \item $H^2_{\Iw, S}(K_\infty, T)$ is $\Lambda_{\Zp}(G)$-torsion.
     \item For each character $\eta$ of $\Delta$, there is a character $\tau$ of $G$ such that $\tau|_\Delta = \eta$ and $H^2(K^S / K, V(\tau)) = 0$.
     \item $H^2(K^S / K_\infty, T \otimes \Qp/\Zp)$ is a cotorsion $\Lambda_{\Zp}(G)$-module.
     \item $H^2(K^S / K_\infty, T \otimes \Qp/\Zp) = 0$.
    \end{enumerate}
   \end{proposition}

   \begin{proof}
    Since $\Delta$ has order prime to $p$, we may assume $\Delta = 1$, so $G \cong \Zp^e$ and $\Lambda = \Lambda_{\Zp}(G)$ is a local integral domain.

    We first show (i) $\Leftrightarrow$ (ii).  By \cite[8.4.8.2 Corollary, (ii)]{nekovar06} we have an isomorphism
    \[ H^2_{\Iw, S}(K_\infty, T) \otimes_{\Lambda} \Zp(\tau) \cong H^2(K^S / K, T(\tau^{-1})).\]

    If $H^2_{\Iw, S}(K_\infty, T)$ is torsion, then it is annihilated by some non-zero $f \in \Lambda$. Since $f \ne 0$, there exists a character $\tau$ such that $f(\tau) \ne 0$; but by the above formula $f(\tau)$ annihilates $H^2(K^S / K, T(\tau^{-1}))$, so $H^2(K^S / K, V(\tau^{-1})) = 0$. Conversely, if $H^2(K^S/K, V(\tau^{-1})) = 0$ for some $\tau$, then $H^2(K^S / K, T(\tau^{-1}))$ is $\Zp$-torsion, so by a form of Nakayama's lemma -- see \cite[Theorem 2]{balisterhowson} -- we can conclude that $H^2_{\Iw, S}(K_\infty, T)$ is a torsion $\Lambda$-module.

    We now show (ii) $\Leftrightarrow$ (iii). We know that $H^2(K^S / K, T(\tau) \otimes \Qp/\Zp)$ is finite if and only if $H^2(K^S / K, V(\tau)) = 0$. From the Hochschild--Serre spectral sequence and Poincar\'e duality for $G$-cohomology we have an isomorphism
    \[ H^2(K^S / K_\infty, T \otimes \Qp/\Zp)^\vee \otimes_{\Lambda} \Zp(\tau) \cong H^2(K^S / K, T(\tau) \otimes \Qp/\Zp)^\vee\]
    and we conclude by the same argument as before.

    To finish the proof, it suffices to show that (iii) $\Rightarrow$ (iv). We claim that the module $H^2(K^S / K_\infty, T \otimes \Qp/\Zp)$ is co-free over $\Lambda$, i.e.~its Pontryagin dual $X = H^2(K^S / K_\infty, T \otimes \Qp/\Zp)^\vee$ is a free $\Lambda$-module; thus if it is cotorsion, it must be zero. For $e = 1$ this is a theorem of Greenberg, cf.~\cite[Proposition 1.3.2]{perrinriou95}, so we shall reduce to this case by induction on $e$.

    Let us choose topological generators $\gamma_1, \dots, \gamma_e$ of $\Gal(K_\infty / K) \cong \Zp^e$, and set $u_i = [\gamma_i] - 1 \in \Lambda$. Then $\Lambda \cong \Zp[[u_1, \dots, u_e]]$ and in particular $(p, u_1, \dots, u_e)$ is a regular sequence for $\Lambda$; so in order to show that $X$ is free, it suffices to show that $X[u_e] = 0$ and $X/u_e X$ is free as a module over $\Lambda / u_e \Lambda$.

    If we let $U$ be the subgroup of $G$ generated by $\gamma_e$, then
    \[ X[u_e] = H^1(U, H^2(K_\infty, T \otimes \Qp/\Zp))^\vee,\]
    and by the Hochschild--Serre exact sequence, $H^1(U, H^2((K_\infty)^U, T \otimes \Qp/\Zp))$ injects into $H^3(K_\infty^{U}, T \otimes \Qp/\Zp)$, which is 0 (since $p$ is odd); and we have
    \[ X / u_e X = H^2((K_\infty)^U, T \otimes \Qp/\Zp)^\vee,\]
    which (by the induction hypothesis) is free over $\Zp[[u_1, \dots, u_{e-1}]]$, so we are done.

   \end{proof}

   To define our module of $p$-adic $L$-functions we will need to assume the following conjecture, which corresponds to the ``conjecture de Leopoldt faible'' of \cite[\S 1.3]{perrinriou95}:

   \begin{conjecture}[Conjecture $\operatorname{Leop}(K_\infty, V)$]
    \label{conj:leopoldt}
    The equivalent conditions of Proposition \ref{prop:leopoldt} hold, for some (and hence every) $\Zp$-lattice $T$ in $V$.
   \end{conjecture}

   Note that if $K_\infty, L_\infty$ are two extensions of $K$ satisfying our conditions, with $K_\infty \subseteq L_\infty$, and $\Gal(L_\infty/K_\infty)$ is torsion-free (hence isomorphic to a product of copies of $\Zp$), then conjecture $\Leop(K_\infty, V)$ implies conjecture $\Leop(L_\infty, V)$, since $\Gal(K_\infty / K)$ and $\Gal(L_\infty / K)$ have the same torsion subgroup and thus condition (ii) of Proposition \ref{prop:leopoldt} for $K_\infty$ implies the corresponding condition for $L_\infty$. It is conjectured that $\Leop(K(\mu_{p^\infty}), V)$ should hold for any $V$, and this is known in many cases; see \cite[Appendix B]{perrinriou95}.

   \begin{example}
    Let $V$ be the 2-dimensional $p$-adic representation of $\cG_\QQ$ associated to a modular form, $K / \QQ$ an imaginary quadratic field, and $K_\infty$ the unique $\Zp^2$-extension of $K$. Then $\operatorname{Leop}(K_\infty, V)$ holds.

    To see this, we use the fact that $\operatorname{Leop}(K_\infty, V)$ is implied by $\operatorname{Leop}(K^{\mathrm{cyc}}, V)$, where $K^{\mathrm{cyc}}$ is the cyclotomic $\Zp$-extension of $K$. However, by Shapiro's lemma the conjecture $\operatorname{Leop}(K^{\mathrm{cyc}}, V)$ is equivalent to $\operatorname{Leop}(\QQ^{\operatorname{cyc}}, V \oplus V(\varepsilon_K))$, where $\QQ^{\operatorname{cyc}}$ is the cyclotomic $\Zp$-extension of $\QQ$ and $\varepsilon_K$ is the quadratic Dirichlet character associated to $K$. The conjectures and $\operatorname{Leop}(\QQ^{\operatorname{cyc}}, V)$ and $\operatorname{Leop}(\QQ^{\operatorname{cyc}}, V(\varepsilon_K))$ follow from \cite[Theorem 12.4]{kato04} applied to $f$ and its twist by $\varepsilon_K$.
   \end{example}

   \begin{corollary}
    If $K$ is totally complex and Conjecture $\operatorname{Leop}(K_\infty, V)$ holds, then the module $H^1_{\Iw, S}(K_\infty, T)$ has well-defined $\Lambda_{\Zp}(G)$-rank, equal to $\tfrac{1}{2}[K : \QQ]d$, where $d=\rank_{\Zp}T$.
   \end{corollary}


 \section{Explicit formulae for Perrin-Riou's p-adic regulator}
  \label{appendix:cyclo}

  In this section, we give the proof of the formulae for the cyclotomic regulator used in the proof of Proposition \ref{thm:explicitformula}. As we work only over $\Qp$ here, we shall write $\DD(-)$ and $\NN(-)$ for $\DD_{\Qp}(-)$ and $\NN_{\Qp}(-)$ respectively.

  Let $V$ be a good crystalline representation of $\cG_{\Qp}$, and $x \in \cH(\Gamma) \otimes_{\Lambda_{\Zp}(\Gamma)} H^1_{\Iw}(\QQ_{p, \infty}, V)$. We write $x_j$ for the image of $x$ in $H^1_{\Iw}(\QQ_{p, \infty}, V(-j))$, and $x_{j, n}$ for the image of $x_j$ in $H^1(\QQ_{p, n}, V(-j))$. If we identify $x$ with its image in $\DD(V)^{\psi = 1}$, then $x_j$ corresponds to the element $x \otimes e_{-j} \in \DD(V)^{\psi = 1} \otimes e_{-j} = \DD(V(-j))^{\psi = 1}$.

  Since $V$ has non-negative Hodge--Tate weights, we may interpret $x$ as an element of the module $\left( \BB^+_{\rig, \Qp}\left[\tfrac1t\right] \otimes \Dcris(V)\right)^{\psi = 1}$.

  We shall assume:
  \[ \tag{$\dag$} x \in \left(\BB^+_{\rig, \Qp} \otimes_{\BB^+_{\Qp}} \NN(V)\right)^{\psi = 1} \subseteq \left( \BB^+_{\rig, \Qp}\left[\tfrac1t\right] \otimes_{\Qp} \Dcris(V)\right)^{\psi = 1}.\]

  This condition is satisfied in the following two situations:
  \begin{itemize}
   \item if $V$ has no quotient isomorphic to $\Qp$, by \cite[Theorem A.3]{berger03};
   \item or if $x$ is in the image of the Iwasawa cohomology over $F_\infty(\mu_{p^\infty})$, by Theorem \ref{thm:unramunivnorms} above.
  \end{itemize}

  We will base our proofs on the work of Berger \cite{berger03}, so we recall the notation of that reference. Let $\partial$ denote the differential operator $(1 + \pi) \tfrac{\mathrm{d}}{\mathrm{d}\pi}$ on $\Brig$. We also use Berger's notation $\partial_V \circ \vp^{-n}$ for the map
  \[  \Brig\left[\tfrac1t\right] \otimes_{\Qp} \Dcris(V) \rTo  \QQ_{p, n}\otimes_{\Qp} \Dcris(V)\]
  which sends $\pi^k\otimes v$ to the constant coefficient of $(\zeta_n \exp(t/p^n) - 1)^k \otimes \vp^{-n}(v)\in \QQ_{p,n}((t))\otimes_{\Qp}\Dcris(V)$.

  For $m \in \ZZ$, define $\Gamma^*(m)$ to be the leading term of the Taylor series expansion of $\Gamma(x)$ at $x = m$ (cf.~\cite[\S 3.3.6]{fukayakato06}); thus
  \[ \Gamma^*(1 + j) = \begin{cases} j! & \text{if $j \ge 0$,} \\ \frac{(-1)^{-j-1}}{(-j-1)!} & \text{if $j \le -1$.}\end{cases}.\]

  \begin{proposition}
   For $x$ satisfying $(\dag)$, let us define
   \[ R_{j, n}(x) = \frac{1}{\Gamma^*(1 + j)} \times
    \begin{cases}
     p^{-n} \partial_{V(-j)}(\vp^{-n}(\partial^{j} x \otimes t^j e_{-j})) & \text{if $n \ge 1$,}\\
     (1 - p^{-1} \vp^{-1}) \partial_{V(-j)}(\partial^{j} x \otimes t^j e_{-j}) & \text{if $n = 0$.}
    \end{cases}
   \]
   Then we have
   \[ R_{j, n}(x) =
    \begin{cases}
     \exp^*_{\QQ_{p, n}, V(-j)^*(1)}(x_{j, n}) & \text{for $j \ge 0$,}\\
     \log_{\QQ_{p, n}, V(-j)}(x_{j, n}) & \text{for $j \le -1$.}
    \end{cases}
   \]
  \end{proposition}

  \begin{proof}
   This result is essentially a minor variation on \cite[Theorem II.10]{berger03}. The case $j \ge 0$ is immediate from Theorem II.6 of \emph{op.cit.} applied with $V$ replaced by $V(-j)$ and $x$ by $x \otimes e_{-j}$, using the formula
   \[ \partial_{V(-j)}(\vp^{-n}(x \otimes e_{-j})) = \frac{1}{j!} \partial_{V(-j)}(\vp^{-n}(\partial^j x \otimes t^{j} e_{-j})).\]

   For the formula when $j \le -1$, we choose an auxilliary integer $h \ge 1$ such that $\Fil^{-h} \Dcris(V) = \Dcris(V)$. The element $\partial^{j} x \otimes t^{j} e_{-j}$ lies in $\left(\Brig \otimes_{\Qp} \Dcris(V(-j))\right)^{\psi = 1}$, by (\dag). Applying Theorem II.3 of \emph{op.cit.} with $V$, $h$ and $x$ replaced by $V(-j)$, $h-j$, and $\partial^{j} x \otimes t^{-j} e_j$, we see that
   \[
    \Gamma^*(j+1) R_{j, n}(x) = \Gamma^*(j - h + 1) \log_{\QQ_{p, n}, V(-j)}
    \left[ \left(\ell_0 \dots \ell_{h-1} x \right)_{j, n}\right].
   \]
   For $x \in \cH(\Gamma) \otimes_{\Lambda_{\Zp}(\Gamma)} H^1_{\Iw}(\QQ_{p, \infty}, V)$, we have
   \[ \left( \ell_r x\right)_{j, n} = (j - r) x_{j, n},\]
   so (since $j \le -1$) we have
   \[ \left( \ell_0 \dots \ell_{h-1} x \right)_{j, n} = (j)(j-1) \dots (j - h + 1) x_{j, n} = \frac{\Gamma^*(j+1)}{\Gamma^*(j - h + 1)} x_{j, n}\]
   as required.
  \end{proof}

  \begin{proposition}
   If $x$ is as above, and $\mathcal{L}^\Gamma_V(x)$ is the unique element of $\cH(\Gamma) \otimes_{\Qp} \Dcris(V)$ such that $\mathcal{L}^\Gamma_V(x) \cdot (1 + \pi) = (1 - \vp) x$, then for any $j \in \ZZ$ we have
   \[ (1 - \vp) \cdot \partial_{V(-j)}(\vp^{-n}(\partial^{j} x \otimes t^j e_{-j})) = \mathcal{L}^\Gamma(x)(\chi^j) \otimes t^j e_{-j},\]
   while for any finite-order character $\omega$ of $\Gamma$ of conductor $n \ge 1$, we have
   \begin{multline*} \left(\sum_{\sigma \in \Gamma / \Gamma_n} \omega(\sigma)^{-1} \sigma\right) \cdot \partial_{V(-j)}(\vp^{-n}(\partial^{j} x \otimes t^j e_{-j})) \\=
    \tau(\omega) \vp^{-n}  \left(\mathcal{L}^\Gamma(x)(\chi^j \omega) \otimes t^j e_{-j}\right).
   \end{multline*}
  \end{proposition}

  \begin{proof}
   We note that
   \[ \mathcal{L}^\Gamma_{V(-j)}(\partial^{j} x \otimes t^j e_{-j}) = \Tw_j(\mathcal{L}^\Gamma_V(x)) \otimes t^{j} e_{-j},\]
   so it suffices to prove the result for $j = 0$. Suppose we have
   \[ x = \sum_{k \ge 0} v_k \pi^k,\quad v_k \in \Dcris(V).\]
   Then
   \[ \partial_V(\vp^{-n}(x)) = \sum_{k \ge 0} \vp^{-n}(v_k) \left(\zeta_{p^n} - 1 \right)^k.\]
   On the other hand
   \[ \partial_V(\vp^{-n}((1-\vp) x)) = \sum_{k \ge 0} \vp^{-n}(v_k) \left(\zeta_{p^n} - 1 \right)^k - \sum_{k \ge 0} \vp^{1-n}(v_k) \left(\zeta_{p^{n-1}} - 1 \right)^k.\]
   Applying the operator $e_\omega = \sum_{\sigma \in \Gamma / \Gamma_n} \omega(\sigma)^{-1} \sigma$, we have for $n \ge 1$
   \[ e_\omega \cdot \partial_V(\vp^{-n}(x)) = e_\omega \cdot \partial_V(\vp^{-n}((1-\vp) x)),\]
   since $e_\omega$ is zero on $\QQ_{p, n-1}((t))$.

   However, since the map $\partial_V \circ \vp^{-n}$ is a homomorphism of $\Gamma$-modules, we have
   \begin{align*}
     e_\omega \cdot \partial_V(\vp^{-n}((1-\vp) x)) &= e_\omega \cdot \partial_V( \mathcal{L}^\Gamma(x) \cdot (1 + \pi))\\
     &= \vp^{-n}(\mathcal{L}^\Gamma(x)) \cdot e_\omega \partial_{\Qp}(\vp^{-n}(1 + \pi))\\
     &= \tau(\omega) \vp^{-n}\left( \mathcal{L}^\Gamma(x)(\omega)\right).
   \end{align*}
   This completes the proof of the proposition for $j = 0$.
  \end{proof}

  \begin{definition}
   Let $x\in H^1_{\Iw}(\QQ_{p,\infty},V)$. If $\eta$ is any continuous character of $\Gamma$, denote by $x_\eta$ the image of $x$ in $H_{\Iw}^1(\QQ_{p,\infty}, V(\eta^{-1}))$. If $n\geq 0$, denote by $x_{\eta,n}$ the image of $x_\eta$ in $H^1(\QQ_{p,n}, V(\eta^{-1}))$.
  \end{definition}

  Thus $x_{\chi^j, n} = x_{j, n}$. in the previous notation. The next lemma is valid for arbitrary de Rham representations of $\cG_{\Qp}$ (with no restriction on the Hodge--Tate weights):

  \begin{lemma}
   For any finite-order character $\omega$ factoring through $\Gamma / \Gamma_n$, with values in a finite extension $E / \Qp$, we have
   \[ \sum_{\sigma \in \Gamma / \Gamma_n} \omega(\sigma)^{-1} \exp^*_{\QQ_{p,n}, V^*(1)}(x_{0, n})^\sigma = \exp^*_{\Qp, V(\omega^{-1})^*(1)}(x_{\omega,0})\]
   and
   \[ \sum_{\sigma \in \Gamma / \Gamma_n} \omega(\sigma)^{-1} \log_{\QQ_{p, n}, V}(x_{0,n})^\sigma = \log_{\Qp, V(\omega^{-1})}(x_{\omega,0})\]
   where we make the identification
   \[ \DdR(V(\omega^{-1})) \cong \left(E \otimes_{\Qp} \QQ_{p, n} \otimes_{\Qp} \Dcris(V)\right)^{\Gamma = \omega}.\]
  \end{lemma}

  \begin{proof}
   This follows from the compatibility of the maps $\exp^*$ and $\log$ with the corestriction maps (cf~\cite[\S\S II.2 \& II.3]{berger03}).
  \end{proof}

  Combining the three results above, we obtain:

  \begin{theorem}
   \label{thm:explicitformulacyclo}
   Let $j \in \ZZ$ and let $x$ satisfy $(\dag)$. Let $\eta$ be a continuous character of $\Gamma$ of the form $\chi^j \omega$, where $\omega$ is a finite-order character of conductor $n$.
   \begin{enumerate}[(a)]
    \item If $j \ge 0$, we have
    \begin{multline*}
     \cL^\Gamma_{V}(x)(\eta)
     = j! \times\\
     \begin{cases}
      (1 - p^j \vp)(1 - p^{-1-j} \vp^{-1})^{-1} \left( \exp^*_{\Qp, V(\eta^{-1})^*(1)}(x_{\eta,0}) \otimes t^{-j} e_j\right) & \text{if $n = 0$,}\\
      \tau(\omega)^{-1} p^{n(1+j)} \vp^n \left(\exp^*_{\Qp, V(\eta^{-1})^*(1)}(x_{\eta,0}) \otimes  t^{-j} e_j\right) & \text{if $n \ge 1$.}
     \end{cases}
    \end{multline*}
    \item If $j \le -1$, we have
    \begin{multline*}
     \cL^\Gamma_{V}(x)(\eta)
     = \frac{(-1)^{-j-1}}{(-j-1)!} \times\\
     \begin{cases}
      (1 - p^j \vp)(1 - p^{-1-j} \vp^{-1})^{-1} \left( \log_{\Qp, V(\eta^{-1})}(x_{\eta,0}) \otimes t^{-j} e_j\right) & \text{if $n = 0$,}\\
      \tau(\omega)^{-1} p^{n(1+j)} \vp^n \left(\log_{\Qp, V(\eta^{-1})}(x_{\eta,0}) \otimes  t^{-j} e_j\right) & \text{if $n \ge 1$.}
     \end{cases}
    \end{multline*}
   \end{enumerate}
   (In both cases, we assume that $(1 - p^{-1-j} \vp^{-1})$ is invertible on $\Dcris(V)$ when $\eta = \chi^j$.)
  \end{theorem}

  From this theorem it is straightforward to deduce a version of Perrin-Riou's explicit reciprocity formula, relating the regulator for $V$ and for $V^*(1)$. We recall from \ref{sect:iwasawacoho} the definition of the Perrin-Riou pairing
  \[ \langle -, - \rangle_{\QQ_{p, \infty}} : H^1_{\Iw}(\QQ_{p, \infty}, V) \times H^1_{\Iw}(\QQ_{p, \infty}, V^*(1)) \to \Lambda_{\Qp}(\Gamma).\]

  Let $h$ be sufficiently large that $V^*(1 + h)$ has Hodge--Tate weights $\ge 0$. Recall that we write $y_{-h}$ for the image of $y$ in $H^1_{\Iw}(\QQ_{p, n}, V^*(1 + h))$.   Define $\cL^\Gamma_{V^*(1)}$ by
  \begin{multline}
   \label{eq:twist1}
   \cL^\Gamma_{V^*(1)}(y) = (\ell_{-1} \ell_{-2} \cdot \ell_{-h})^{-1} \Tw_{-h} \left(\cL_{V^*(1 + h)}(y_{-h})\right) \otimes t^h e_{-h} \\ \in \operatorname{Frac} \cH_{\Qp}(\Gamma) \otimes \Dcris(V^*(1));
   \end{multline}
  note that this definition is independent of the choice of $h \gg 0$. Write $\langle \cdot, \cdot \rangle_{\cris, V}$ for the natural pairing $\Dcris(V) \times \Dcris(V^*(1)) \to \Dcris(\Qp(1)) \cong \Qp$. We extend the crystalline pairing $\Lambda_{\Zp}(\Gamma)$-linearly in the first argument and antilinearly in the second argument.

  \begin{theorem}
   \label{thm:cycloreciprocity}
   For all $x\in H^1_{\Iw}(\QQ_{p,\infty},V)$ and $y\in H^1_{\Iw}(\QQ_{p,\infty},V^*(1))$, we have
   \[ \left\langle \cL_V(x), \cL_{V^*(1)}(y)\right\rangle_{\cris, V} = -\sigma_{-1} \cdot \ell_0 \cdot \langle x, y \rangle_{\QQ_{p, \infty}, V},\]
   where $\sigma_{-1}$ is the unique element of $\Gamma$ such that $\chi(\sigma_{-1}) = -1$.
  \end{theorem}

 \begin{proof}
  By Theorem \ref{thm:explicitformulacyclo} (a), for $j \ge 1+h$ we have
  \[
   \cL_V(x)(\chi^j) = j! (1 - p^j \vp)(1 - p^{-1-j} \vp^{-1})^{-1} \left( \exp^*_{0, V^*(1+j)}(x_{j,0}) \otimes t^{-j} e_j\right)\]
  and
  \begin{multline*}
   \cL_{V^*(1 + h)}(y_{-h})(\chi^{h-j}) \otimes t^h e_{-h}  =\frac{(-1)^{j-h-1}}{(j-h-1)!} \times\\  (1 - p^{-j} \vp)(1 - p^{j-1} \vp^{-1})^{-1} \left( \log_{\Qp, V^*(1 + j)} (y_{-j,0}) \otimes t^{j} e_{-j}\right).
  \end{multline*}

   Hence we have
  \begin{align*}
    \big\langle \cL_V(x)(\chi^j), & \cL_{V^*(1 + h)}(y_{-h})(\chi^{h-j}) \otimes t^h e_{-h}\big\rangle_{\cris, V} \\
   & = \frac{(-1)^{h-j-1} j!}{(j-h-1)!} \langle \exp^*_{0, V^*(1+j)}(x_{j,0}), \log_{\Qp, V^*(1 + j)} (y_{-j,0}) \rangle_{\cris, V(-j)}\\
   & = \frac{(-1)^{h-j-1} j!}{(j-h-1)!} \langle x_{j,0}, y_{-j,0}\rangle_{\Qp, V(-j)}\\
   & = (-1)^{h+1} \left[ \sigma_{-1} \cdot (\ell_0 \dots \ell_{h}) \cdot \langle x, y \rangle_{\QQ_{p, \infty}, V} \right] (\chi^j).
  \end{align*}

  Using the definition of $\cL^\Gamma_{V^*(1)}$ as in \eqref{eq:twist1}, this relation takes the more pleasing form
  \[ \left\langle \cL_V(x), \cL_{V^*(1)}(y)\right\rangle_{\cris, V} = -\sigma_{-1} \cdot \ell_0 \cdot \langle x, y \rangle_{\QQ_{p, \infty}, V}.\]
 \end{proof}


 \section{Functions of two p-adic variables}

Let $p$ be a prime. We let $L$ be a complete discretely valued subfield of $\Cp$, and let $v_p$ denote the $p$-adic valuation on $L$, normalised in the usual fashion, so $v_p(p) = 1$.

\subsection{Functions and distributions of one variable}

We recall the theory in the one-variable case, as presented in \cite{colmez10}. Let $h \in \mathbb{R}$, $h \ge 0$.

Let $f$ be a function $\Zp \to L$. We say $f$ has order $h$ if, informally, it may be approximated by a Taylor series of degree $[h]$ at every point with an error term of order $h$. More precisely, $f$ has order $h$ if there exist functions $f^{(j)}, 0 \le j \le [h]$, such that the quantity
\[ \ve_f(x,y) = f(x + y) - \sum_{j=0}^{[h]} \frac{f^{(j)}(x) y^j}{j!},\]
satisfies
\[ \sup_{x \in \Zp, y \in p^n \Zp} v_p \left( \ve_f(x, y)\right) -hn \to \infty\]
as $n \to \infty$. (It is clear that this determines the functions $f^{(0)}, \dots, f^{(j)}$ uniquely.)

We write $C^h(\Zp, L)$ for the space of such functions, with a Banach space structure given by the valuation
\[ v_{C^h}(f) = \inf\left( \inf_{0 \le j \le [h], x \in \Zp} v_p(f^{(j)}(x)), \inf_{x,y \in \Zp} v_p(\ve_f(x,y)) - h v_p(y)\right).\]
We define the space $D^h(\Zp, L)$ of \emph{distributions of order h} to be the continuous dual of $C^h(\Zp, L)$.

Then we have the following celebrated theorem, due to Mahler \cite{mahler58} for $h = 0$ and to Amice \cite{amice64} for $h > 0$:

\begin{enumerate}
 \item[(1)] The space $C^h(\Zp, L)$ has a Banach space basis given by the functions
 \[ x \mapsto p^{[h \ell(n)]} \binom{x}{n}\]
 for $n \ge 0$, where $\ell(n)$ is, as in \S 1.3.1 of \cite{colmez10}, the smallest integer $m$ such that $p^m > n$.
 \item[(2)] The space $LP^{N}(\Zp, L)$ of $L$-valued locally polynomial functions of degree $N$ is dense in $C^h(\Zp, L)$ for any $N \ge [h]$, and a linear functional
 \[ \mu: LP^{N}(\Zp, L) \to L\]
 extends continuously to a distribution of order $h$ if and only there is a constant $C$ such that we have
 \[ v_p\left(\int_{x \in a + p^n \Zp} \left( \frac{x - a}{p^n} \right)^k \, \mathrm{d}\mu\right) \ge C - hn\]
 for all $a \in \Zp$, $n \in \mathbb{N}$ and $0 \le k \le N$.
\end{enumerate}

A modern account of this theorem is given in \cite[\S\S 1.3, 1.5]{colmez10}.

\subsection{The two-variable case}
\label{sect:order}

We now consider functions of two variables. For $a, b \ge 0$, we define the space
\[ C^{(a,b)}(\Zp^2, L) := C^a(\Zp, L) \htimes_L C^b(\Zp, L),\]
with its natural completed tensor product topology. We regard this as a space of functions on $\Zp^2$ in the obvious way, and refer to these as the $L$-valued functions on $\Zp^2$ of order $(a, b)$. It is clear that $C^{(0, 0)}(\Zp^2, L)$ is simply the space of continuous $L$-valued functions on $\Zp^2$, and that if $a' \ge a$ and $b' \ge b$, then $C^{(a',b')}(\Zp^2, L)$ is dense in $C^{(a,b)}(\Zp^2, L)$. Moreover, for any $(a,b)$ the space $LA(\Zp^2, L)$ of locally analytic functions on $\Zp^2$ is a dense subspace of $C^{(a,b)}(\Zp^2, L)$.

Note that any choice of Banach space bases for the two factors in the tensor product gives a Banach space basis for $C^{(a,b)}(\Zp^2, L)$. In particular, from (1) above we have a Banach basis given by the functions
\[ c_{n_1,n_2} : (x_1,x_2) \mapsto p^{[a \ell(n_1)] + [b \ell(n_2)]} \binom{x_1}{n_1}\binom{x_2}{n_2}.\]

The following technical proposition will be useful to us in the main text:

\begin{proposition}
\label{prop:changevar}
 For any $h \ge 0$, the space $C^{(0, h)}(\Zp^2, L)$ is invariant under pullback via the map $\Phi: (x, y) \to (x, ax + y)$, for any $a \in \Zp$.
\end{proposition}

\begin{proof}
 It suffices to show that $\Phi^*(c_{n_1,n_2})$ can be written as a convergent series in terms of the functions $c_{m_1, m_2}$ with uniformly bounded coefficients. We find that
 \begin{align*}
  \Phi^*(c_{n_1,n_2})(x_1, x_2) &= p^{[h \ell(n_2)]} \binom{x_1}{n_1} \binom{ax_1 + x_2}{n_2}\\
  &= \sum_{i=0}^{n_1} p^{[h \ell(n_2)]} \binom{x_1}{n_1} \binom{ax_1}{n_2 - i} \binom{x_2}{i}.
 \end{align*}
 The functions $x_1 \mapsto \binom{x_1}{n_1} \binom{ax_1}{n_2 - i}$ are continuous $\Zp$-valued functions on $\Zp$, and hence the coefficients of their Mahler expansions are integral; and since the function $\ell(n)$ is increasing, we see that the coefficients of $\Phi^*(c_{n_1,n_2})$ in this basis are in fact bounded by 1.
   \end{proof}

Dually, we define a distribution of order $(a,b)$ to be an element of the dual of $C^{(a,b)}(\Zp^2, L)$; the space $D^{(a,b)}(\Zp^2, L)$ of such distributions is canonically isomorphic to the completed tensor product $D^a(\Zp, L) \htimes_L D^b(\Zp, L)$.

An analogue of (2) above is also true for these spaces. Let us write $LP^{(N_1, N_2)}(\Zp^2, L)$ for the space of functions on $\Zp^2$ which are locally polynomial of degree $\le N_1$ in $x_1$ and of degree $\le N_2$ in $x_2$; that is, the algebraic tensor product $LP^{N_1}(\Zp, L) \otimes_L LP^{N_2}(\Zp, L)$.

\begin{proposition}
 Suppose $N_1 \ge [a]$ and $N_2 \ge [b]$. Then the subspace $LP^{(N_1, N_2)}(\Zp^2, L)$ is dense in $C^{(a,b)}(\Zp^2, L)$, and a linear functional on $LP^{(N_1, N_2)}(\Zp^2, L)$ extends to an element of $D^{(a, b)}(\Zp, L)$ if and only if there is a constant $C$ such that
\begin{multline}
\label{eq:ordercondition}
  v_p \left(\int_{(x_1, x_2) \in (a_1 + p^{n_1}\Zp) \times (a_2 + p^{n_2}\Zp)} \left( \frac{x_1 - a_1}{p^{n_1}}\right)^{k_1} \left( \frac{x_2 - a_2}{p^{n_2}}\right)^{k_2}\,\mathrm{d}\mu\right) \\ \ge C - a n_1 - b n_2
 \end{multline}
for all $(a_1, a_2) \in \Zp^2$, $(n_1, n_2) \in \NN^2$, $0 \le k_1 \le N_1$ and $0 \le k_2 \le N_2$.
\end{proposition}

The proof of this result is virtually identical to the 1-variable case, so we shall not give the full details here.

In particular, if $a,b < 1$, we may take $N_1 = N_2 = 0$, and a distribution of order $(a, b)$ is uniquely determined by its values on locally constant functions, or equivalently, by its values on the indicator functions of open subsets of $\Zp^2$. Conversely, a finitely-additive function $\mu$ on open subsets of $\Zp^2$ defines a distribution of order $(a,b)$ if and only if there is $C$ such that
\[ v_p \mu\left( (a_1 + p^{n_1}\Zp) \times (a_2 + p^{n_2}\Zp) \right) \ge C - a n_1 - b n_2.\]

The following is easily verified:

\begin{proposition}
\label{prop:orderofproduct}
 The convolution of distributions of order $(a,b)$ and $(a', b')$ has order $(a + a', b + b')$.
\end{proposition}

It is important to note that the spaces of functions and of distributions of order $(a,b)$ depend on a choice of coordinates; they are not invariant under automorphisms of $\Zp^2$, \emph{even if} $a = b$. However, dualising Proposition \ref{prop:changevar} above, the space of distributions of order $(0, h)$ is invariant under automorphisms preserving the subgroup $(0, \Zp)$.

\begin{remark}
 One can also define a function $f:\Zp^2 \to L$ to be of order $h$, for a single non-negative real $h$, if $f$ has a Taylor expansion of degree $[h]$ at every point, with the error term $\ve(x,y)$ (defined as above) satisfying
 \[ \inf_{x \in \Zp^2, y \in p^n \Zp^2} v_p \ve(x,y) - hn \to \infty.\]
 This definition \emph{is} invariant under automorphisms of $\Zp$ (and indeed under arbitrary morphisms of locally $\Qp$-analytic manifolds). However it is not so convenient for us, since locally constant functions are only dense for $h < 1$, and a finitely-additive function on open subsets extends to a linear functional on this space if we can find a $C$ such that
 \begin{equation}
\label{eq:badordercondition}
v_p \mu\left( a + p^n \Zp^2 \right) \ge C - nh.
\end{equation}
 The requirement that this be satisfied, for some $h < 1$, is much stronger than the requirement that \eqref{eq:ordercondition} is satisfied for some $a, b < 1$.
\end{remark}

We shall also use the concept of distributions of order $(a, b)$ on a slightly larger class of group: if we have an abelian $p$-adic Lie group $G$, and an open subgroup $H$ with distinguished subgroups $H_1,H_2$ such that $H = H_1 \times H_2$ and $H_1 \cong H_2 \cong \Zp$, then we may define a distribution on $G$ to have order $(a, b)$ if its restriction to every coset of $H$ has order $(a, b)$ in the above sense. Note that this does not depend on a choice of generators of the groups $H_i$, but it does depend on the choice of the subgroups $H_1, H_2$; so when there is a possibility of ambiguity we shall write ``order $(a,b)$ with respect to the subgroups $H_1, H_2$''.

Note that an application of Proposition \ref{prop:changevar} shows that a distribution has order $(0, h)$ with respect to the subgroups $(H_1, H_2)$ if and only if it has order $(0, h)$ with respect to $(H_1', H_2)$ for any other subgroup $H_1'$ complementary to $H_2$; that is, in this special case the definition of ``order $(0, h)$'' depends only on the choice of $H_2$.


\section*{Acknowledgements}

We are grateful to Eric Pickett for alerting us to the results on integral normal bases contained in his paper \cite{pickett10}. The main idea of this paper was developed during our visit to Salt Lake City in February 2011; we would like to thank Wieslawa Niziol and Pierre Colmez for their hospitality. We would also like to thank John Coates, Pierre Colmez, Henri Darmon, Karl Rubin and Otmar Venjakob for their interest and for some helpful discussions, and Ivan Fesenko for raising the question of whether Perrin-Riou's local reciprocity formula holds in this setting. We also thank the referee for a number of helpful suggestions and corrections.
\newcommand{\etalchar}[1]{$^{#1}$}
\newcommand{\noopsort}[1]{} \newcommand{\singleletter}[1]{#1} \def\cprime{$'$}
\providecommand{\bysame}{\leavevmode\hbox to3em{\hrulefill}\thinspace}
\providecommand{\MR}[1]{\relax}
\renewcommand{\MR}[1]{%
 MR \href{http://www.ams.org/mathscinet-getitem?mr=#1}{#1}.
}
\providecommand{\href}[2]{#2}
\newcommand{\articlehref}[2]{\href{#1}{#2}}


\begin{thebibliography}{CFK{\etalchar{+}}05}

\bibitem[Ami64]{amice64}
Yvette Amice, \emph{Interpolation {$p$}-adique}, Bull. Soc. Math. France
  \textbf{92} (1964), 117--180. \MR{0188199}

\bibitem[BH97]{balisterhowson}
Paul Balister and Susan Howson,
  \articlehref{http://www.intlpress.com/AJM/AJM-v01.php\#AJM-1-2}{\emph{Note on
  {N}akayama's lemma for compact {$\Lambda$}-modules}}, Asian J. Math.
  \textbf{1} (1997), no.~2, 224--229. \MR{1491983}

\bibitem[Bel11]{bellaiche11}
Jo{\"e}l Bella{\"i}che, \emph{Computation of {$p$}-adic {$L$}-functions of critical {CM}
  forms}, preprint, 2011.

\bibitem[Ben00]{benois00}
Denis Benois,
  \articlehref{http://dx.doi.org/10.1215/S0012-7094-00-10422-X}{\emph{On
  {I}wasawa theory of crystalline representations}}, Duke Math. J. \textbf{104}
  (2000), no.~2, 211--267. \MR{1773559}

\bibitem[Ber03]{berger03}
Laurent Berger,
  \articlehref{http://www.mathematik.uni-bielefeld.de/documenta/vol-kato/berger.dm.html}{\emph{Bloch
  and {K}ato's exponential map: three explicit formulas}}, Doc. Math. Extra
  Vol. \textbf{3} (2003), 99--129, Kazuya Kato's fiftieth birthday.
  \MR{2046596}

\bibitem[Ber04]{berger04}
\bysame,
  \articlehref{http://dx.doi.org/10.1112/S0010437X04000879}{\emph{Limites de
  repr{\'e}sentations cristallines}}, Compos. Math. \textbf{140} (2004), no.~6,
  1473--1498. \MR{2098398}

\bibitem[BDP13]{BDP13}
Massimo Bertolini, Henri Darmon, and Kartik Prasanna,
  \articlehref{http://dx.doi.org/10.1215/00127094-2142056}{\emph{Generalized
  {H}eegner cycles and {$p$}-adic {R}ankin {$L$}-series}}, Duke Math. J.
  \textbf{162} (2013), no.~6, 1033--1148, With an appendix by Brian Conrad.
  \MR{3053566}

\bibitem[BK90]{blochkato90}
Spencer Bloch and Kazuya Kato,
  \articlehref{http://www.springer.com/birkhauser/mathematics/book/978-0-8176-4566-3}{\emph{{$L$}-functions
  and {T}amagawa numbers of motives}}, The {G}rothendieck {F}estschrift,
  {V}ol.\ {I}, Progr. Math., vol.~86, Birkh{\"a}user Boston, Boston, MA, 1990,
  pp.~333--400. \MR{1086888}

\bibitem[Bra11]{brakocevic}
Miljan Brako\v{c}evi\'c,
  \articlehref{http://dx.doi.org/10.1093/imrn/rnq275}{\emph{Anticyclotomic
  {$p$}-adic {$L$}-function of central critical {R}ankin--{S}elberg
  {$L$}-value}}, Int. Math. Res. Notices \textbf{2011} (2011), no.~21,
  4967--5018. \MR{2852303}

\bibitem[CC99]{cherbonniercolmez99}
Fr{\'e}d{\'e}ric Cherbonnier and Pierre Colmez,
  \articlehref{http://dx.doi.org/10.1090/S0894-0347-99-00281-7}{\emph{Th{\'e}orie
  d'{I}wasawa des repr{\'e}sentations {$p$}-adiques d'un corps local}}, J.
  Amer. Math. Soc. \textbf{12} (1999), no.~1, 241--268. \MR{1626273}

\bibitem[CFK{\etalchar{+}}05]{cfksv}
John Coates, Takako Fukaya, Kazuya Kato, Ramdorai Sujatha, and Otmar Venjakob,
  \articlehref{http://dx.doi.org/10.1007/s10240-004-0029-3}{\emph{The
  {$\operatorname{GL}_2$} main conjecture for elliptic curves without complex
  multiplication}}, Publ. Math. IH{\'E}S \textbf{101} (2005), 163--208.
  \MR{2217048}

\bibitem[Col79]{coleman79}
Robert Coleman,
  \articlehref{http://dx.doi.org/10.1007/BF01390028}{\emph{Division values in
  local fields}}, Invent. Math. \textbf{53} (1979), no.~2, 91--116.
  \MR{0560409}

\bibitem[Cz98]{colmez98}
Pierre Colmez, \articlehref{http://dx.doi.org/10.2307/121003}{\emph{Th{\'e}orie
  d'{I}wasawa des repr{\'e}sentations de de {R}ham d'un corps local}}, Ann. of
  Math. (2) \textbf{148} (1998), no.~2, 485--571. \MR{1668555}

\bibitem[Cz10]{colmez10}
\bysame,
  \articlehref{http://smf4.emath.fr/Publications/Asterisque/2010/330/html/smf_ast_330_13-59.php}{\emph{Fonctions
  d'une variable {$p$}-adique}}, Ast{\'e}risque \textbf{330} (2010), 13--59.
  \MR{2642404}

\bibitem[DS87]{deshalit87}
Ehud De~Shalit, \emph{Iwasawa theory of elliptic curves with complex
  multiplication}, Perspectives in Mathematics, vol.~3, Academic Press Inc.,
  Boston, MA, 1987. \MR{917944}

\bibitem[Del73]{deligne73}
Pierre Deligne, \emph{Les constantes des \'equations fonctionnelles des
  fonctions {$L$}}, Modular functions of one variable, {II} ({P}roc.
  {I}nternat. {S}ummer {S}chool, {U}niv. {A}ntwerp, {A}ntwerp, 1972), Springer,
  Berlin, 1973, pp.~501--597. Lecture Notes in Math., Vol. 349. \MR{0349635}

\bibitem[Eme04]{emerton04}
Matthew Emerton, \emph{Locally analytic vectors in representations of locally
  {$p$}-adic analytic groups}, Mem. Am. Math. Soc. (to appear), 2004.

\bibitem[Fon90]{fontaine90}
Jean-Marc Fontaine,
  \articlehref{http://dx.doi.org/10.1007/978-0-8176-4575-5}{\emph{Repr{\'e}sentations
  {$p$}-adiques des corps locaux. {I}}}, The Grothendieck Festschrift, Vol.\
  II, Progr. Math., vol.~87, Birkh{\"a}user Boston, Boston, MA, 1990,
  pp.~249--309. \MR{1106901}

\bibitem[Fon94a]{fontaine94b}
\bysame, \emph{Repr{\'e}sentations {$p$}-adiques semi-stables}, Ast{\'e}risque
  \textbf{223} ({\noopsort{b}}1994), 113--184, P{\'e}riodes $p$-adiques
  (Bures-sur-Yvette, 1988). \MR{1293972}

\bibitem[Fon94b]{fontaine94c}
\bysame, \emph{Repr\'esentations {$\ell$}-adiques potentiellement
  semi-stables}, Ast{\'e}risque \textbf{223} ({\noopsort{c}}1994), 321--347,
  P{\'e}riodes $p$-adiques (Bures-sur-Yvette, 1988). \MR{1293977}

\bibitem[FK06]{fukayakato06}
Takako Fukaya and Kazuya Kato, \emph{A formulation of conjectures on {$p$}-adic
  zeta functions in noncommutative {I}wasawa theory}, Proceedings of the {S}t.
  {P}etersburg {M}athematical {S}ociety. {V}ol. {XII} (Providence, RI), Amer.
  Math. Soc. Transl. Ser. 2, vol. 219, Amer. Math. Soc., 2006, pp.~1--85.
  \MR{2276851}

\bibitem[GS93]{greenbergstevens93}
Ralph Greenberg and Glenn Stevens,
  \articlehref{http://dx.doi.org/10.1007/BF01231294}{\emph{{$p$}-adic
  {$L$}-functions and {$p$}-adic periods of modular forms}}, Invent. Math.
  \textbf{111} (1993), no.~2, 407--447. \MR{1198816}

\bibitem[Har87]{haran87}
Shai M~J Haran,
  \articlehref{http://www.numdam.org/item?id=CM_1987__62_1_31_0}{\emph{{$p$}-adic
  {$L$}-functions for modular forms}}, Compos. Math. \textbf{62} (1987), no.~1,
  31--46. \MR{892149}

\bibitem[Hid88]{hida88}
Haruzo Hida,
  \articlehref{http://www.numdam.org/item?id=AIF_1988__38_3_1_0}{\emph{A
  {$p$}-adic measure attached to the zeta functions associated with two
  elliptic modular forms. {II}}}, Ann. Inst. Fourier (Grenoble) \textbf{38}
  (1988), no.~3, 1--83. \MR{976685}

\bibitem[IP06]{iovitapollack06}
Adrian Iovita and Robert Pollack,
  \articlehref{http://dx.doi.org/10.1515/CRELLE.2006.069}{\emph{Iwasawa theory
  of elliptic curves at supersingular primes over {$\mathbb{Z}_p$}-extensions
  of number fields}}, J. reine angew. Math. \textbf{598} (2006), 71--103.
  \MR{2270567}

\bibitem[Kat93]{kato93}
Kazuya Kato, \articlehref{http://dx.doi.org/10.1007/BFb0084729}{\emph{Lectures
  on the approach to {I}wasawa theory for {H}asse-{W}eil {$L$}-functions via
  {$B_{\mathrm{dR}}$}. {I}}}, Arithmetic algebraic geometry ({T}rento, 1991),
  Lecture Notes in Math., vol. 1553, Springer, Berlin, 1993, pp.~50--163.
  \MR{1338860}

\bibitem[Kat04]{kato04}
\bysame, \emph{{$p$}-adic {H}odge theory and values of zeta functions of
  modular forms}, Ast{\'e}risque \textbf{295} (2004), ix, 117--290,
  Cohomologies $p$-adiques et applications arithm{\'e}tiques. III. \MR{2104361}

\bibitem[Kim11]{kim-preprint}
Byoung~Du Kim,
  \articlehref{http://homepages.ecs.vuw.ac.nz/~bdkim/}{\emph{Two-variable
  {$p$}-adic {$L$}-functions of modular forms for non-ordinary primes}},
  preprint, September 2011.

\bibitem[KM76]{knudsenmumford76}
Finn~Faye Knudsen and David Mumford,
  \articlehref{http://www.mscand.dk/article.php?id=2311}{\emph{The projectivity
  of the moduli space of stable curves. {I}. {P}reliminaries on ``det'' and
  ``{D}iv''}}, Math. Scand. \textbf{39} (1976), no.~1, 19--55. \MR{0437541}

\bibitem[LLZ10]{leiloefflerzerbes10}
Antonio Lei, David Loeffler, and Sarah~Livia Zerbes,
  \articlehref{http://www.intlpress.com/AJM/AJM-v14.php\#AJM-14-4}{\emph{Wach
  modules and {I}wasawa theory for modular forms}}, Asian J. Math. \textbf{14}
  (2010), no.~4, 475--528. \MR{2774276}

\bibitem[LLZ11]{leiloefflerzerbes11}
\bysame, \articlehref{http://dx.doi.org/10.2140/ant.2011.5.1095}{\emph{Coleman
  maps and the $p$-adic regulator}}, Algebra \& Number Theory \textbf{5}
  (2011), no.~8, 1095--1131. \MR{2948474}

\bibitem[LLZ13]{leiloefflerzerbes12}
\bysame,
  \articlehref{http://dx.doi.org/10.1007/s11856-013-0020-0}{\emph{Critical
  slope p-adic {L}-functions of {CM} modular forms}}, Israel J. Math.
  \textbf{198} (2013), no.~1, 261--282, \path{arXiv:1112.5579}. \MR{3096640}

\bibitem[Loe13]{loeffler13}
David Loeffler, \articlehref{http://arxiv.org/abs/1304.4042}{\emph{P-adic
  integration on ray class groups and non-ordinary p-adic {L}-functions}},
  Proceedings of the conference ``Iwasawa 2012'' (to appear), 2013,
  \path{arXiv:1304.4042}.

\bibitem[Mah58]{mahler58}
K.~Mahler, \emph{An interpolation series for continuous functions of a
  {$p$}-adic variable}, J. Reine Angew. Math. \textbf{199} (1958), 23--34.
  \MR{0095821}

\bibitem[Nek06]{nekovar06}
Jan Nekov{\'a}{\v{r}}, \emph{Selmer complexes}, Ast{\'e}risque (2006), no.~310,
  viii+559. \MR{2333680}

\bibitem[Neu99]{neukirch99}
J{\"u}rgen Neukirch, \emph{Algebraic number theory}, Grundlehren der
  Mathematischen Wissenschaften [Fundamental Principles of Mathematical
  Sciences], vol. 322, Springer-Verlag, Berlin, 1999. \MR{1697859}

\bibitem[NSW00]{nsw}
J{\"u}rgen Neukirch, Alexander Schmidt, and Kay Wingberg, \emph{Cohomology of
  number fields}, Grundlehren der Mathematischen Wissenschaften [Fundamental
  Principles of Mathematical Sciences], vol. 323, Springer-Verlag, Berlin,
  2000. \MR{1737196}

\bibitem[Pas05]{pasol05}
Vicentiu Pasol, \emph{P-adic modular symbols attached to {C.} {M.} forms},
  Ph.D. thesis, Boston University, 2005. \MR{2707730}

\bibitem[PR88]{perrinriou88}
Bernadette Perrin-Riou,
  \articlehref{http://dx.doi.org/10.1112/jlms/s2-38.1.1}{\emph{Fonctions {$L$}
  {$p$}-adiques associ{\'e}es {\`a} une forme modulaire et {\`a} un corps
  quadratique imaginaire}}, J. London Math. Soc. \textbf{38} (1988), no.~1,
  1--32. \MR{949078}

\bibitem[PR90]{perrinriou90}
\bysame, \emph{Th\'eorie d'{I}wasawa {$p$}-adique locale et globale}, Invent.
  Math. \textbf{99} (1990), no.~2, 247--292.

\bibitem[PR92]{perrinriou92}
\bysame, \articlehref{http://dx.doi.org/10.1007/BF01232022}{\emph{Th\'eorie
  d'{I}wasawa et hauteurs {$p$}-adiques}}, Invent. Math. \textbf{109} (1992),
  no.~1, 137--185. \MR{1168369}

\bibitem[PR94]{perrinriou94}
\bysame, \articlehref{http://dx.doi.org/10.1007/BF01231755}{\emph{Th{\'e}orie
  d'{I}wasawa des repr{\'e}sentations {$p$}-adiques sur un corps local}},
  Invent. Math. \textbf{115} (1994), no.~1, 81--161, With an appendix by
  Jean-Marc Fontaine. \MR{1248080}

\bibitem[PR95]{perrinriou95}
\bysame,
  \articlehref{http://smf4.emath.fr/Publications/Asterisque/1995/229/html/smf_ast_229.html}{\emph{Fonctions
  {$L$} {$p$}-adiques des repr{\'e}sentations {$p$}-adiques}}, Ast{\'e}risque
  \textbf{229} (1995), 1--198. \MR{1327803}

\bibitem[PR03]{perrinriou03}
\bysame,
  \articlehref{http://projecteuclid.org/euclid.em/1067634729}{\emph{Arithm{\'e}tique
  des courbes elliptiques {\`a} r{\'e}duction supersinguli{\`e}re en {$p$}}},
  Experiment. Math. \textbf{12} (2003), no.~2, 155--186. \MR{2016704}

\bibitem[Pic10]{pickett10}
Erik~Jarl Pickett,
  \articlehref{http://dx.doi.org/10.1142/S1793042110003654}{\emph{Construction
  of self-dual integral normal bases in abelian extensions of finite and local
  fields}}, Int. J. Number Theory \textbf{6} (2010), no.~7, 1565--1588.
  \MR{2740722}

\bibitem[ST02]{schneiderteitelbaum02}
P.~Schneider and J.~Teitelbaum,
  \articlehref{http://0-dx.doi.org.pugwash.lib.warwick.ac.uk/10.1007/BF02784538}{\emph{Banach
  space representations and {I}wasawa theory}}, Israel J. Math. \textbf{127}
  (2002), 359--380. \MR{1900706}

\bibitem[Sem88]{semaev88}
I.~A. Semaev, \emph{Construction of polynomials, irreducible over a finite
  field, with linearly independent roots}, Mat. Sb. (N.S.) \textbf{135(177)}
  (1988), no.~4, 520--532, 560. \MR{942137}

\bibitem[Wac96]{wach96}
Nathalie Wach, \emph{Repr{\'e}sentations {$p$}-adiques potentiellement
  cristallines}, Bull. Soc. Math. France \textbf{124} (1996), no.~3, 375--400.
  \MR{1415732}

\bibitem[Yag82]{yager82}
Rodney Yager, \articlehref{http://dx.doi.org/10.2307/1971398}{\emph{On two
  variable {$p$}-adic {$L$}-functions}}, Ann. of Math. (2) \textbf{115} (1982),
  no.~2, 411--449. \MR{0647813}

\end{thebibliography}
\end{document}